\documentclass[11pt, 3p]{elsarticle}

\usepackage{lineno}
\modulolinenumbers[5]

\usepackage[short]{optidef}
\usepackage{amsmath,amssymb,amsthm}
\usepackage{graphicx}
\usepackage{subcaption,booktabs}
\usepackage{tabularx}
\usepackage{comment}
\usepackage{hyperref}
\usepackage{adjustbox}
\usepackage{longtable}
\usepackage{multirow}
\usepackage{multicol}
\usepackage{enumitem}
\usepackage{array}
\usepackage{xcolor}
\usepackage{todonotes}
\usepackage[ruled]{algorithm2e}
\usepackage{background}

\journal{}

\newtheorem{assumption}{Assumption}
\newtheorem{proposition}{Proposition}
\newtheorem{corollary}{Corollary}
\newtheorem{lemma}{Lemma}
\theoremstyle{remark}

\usepackage{tikz}
\usetikzlibrary{arrows.meta, shapes,calc,patterns,decorations.pathmorphing}
\usepackage{pgfplots}
\pgfplotsset{compat=1.16}
\usepackage{subfloat}

\tikzstyle{none}=[]
\tikzstyle{decision} = [diamond, aspect=2, draw, thick, text width=8em, text badly centered, inner sep=2pt]
\tikzstyle{decision2} = [diamond, aspect=3, draw, thick, text width=17em, text badly centered, inner sep=2pt]
\tikzstyle{io} = [trapezium, draw=black, thick, trapezium left angle=70, trapezium right angle=110, text width=17em, text centered]
\tikzstyle{start} = [rectangle, draw=black, thick, text width=17em, text centered, rounded corners, inner sep=5pt]
\tikzstyle{stop} = [rectangle, draw=black, thick, text width=5.5em, text centered, rounded corners, inner sep=5pt]
\tikzstyle{block} = [rectangle, draw=black, thick, text width=17em, text centered, inner sep=5pt]
\tikzstyle{block2} = [rectangle,
    text width=4em, text centered, inner sep=5pt]
\tikzstyle{line} = [draw, thick, -latex]

\biboptions{authoryear}

\usepackage{setspace}
\begin{document}

\SetBgContents{Published at \url{https://doi.org/10.1016/j.ejor.2026.07.045}}      
\SetBgPosition{current page.center}
\SetBgAngle{0}                                    
\SetBgColor{gray}                                 
\SetBgScale{1.5}                                  
\SetBgHshift{0}                                   
\SetBgVshift{9cm} 

\begin{frontmatter}

\title{Strong bounds for large-scale Minimum Sum-of-Squares Clustering}

\author{Anna Livia Croella}
\ead{annalivia.croella@unimercatorum.it}
\address{Department of Engineering and Science, \\ Universitas Mercatorum, Piazza Mattei 10, 00186, Italy}
\author{Veronica Piccialli\corref{cor1}}
\ead{veronica.piccialli@uniroma1.it}
\address{Department of Computer, Control and Management Engineering ''Antonio Ruberti'', \\ Sapienza University of Rome, Via Ariosto 25, 00185, Italy}
\author{Antonio M. Sudoso}
\ead{antoniomaria.sudoso@uniroma1.it}
\address{Department of Computer, Control and Management Engineering ''Antonio Ruberti'', \\ Sapienza University of Rome, Via Ariosto 25, 00185, Italy}
\cortext[cor1]{Corresponding author}

\begin{abstract}
Clustering is a fundamental technique in data analysis and machine learning, used to group similar data points together. Among various clustering methods, the Minimum Sum-of-Squares Clustering (MSSC) is one of the most widely used. MSSC aims to minimize the total squared Euclidean distance between data points and their corresponding cluster centroids. Due to the unsupervised nature of clustering, achieving global optimality is crucial but computationally challenging. The complexity of finding the global solution increases exponentially with the number of data points, making exact methods impractical on large-scale datasets. Even obtaining strong lower bounds on the optimal MSSC objective value is computationally prohibitive, making it difficult to assess the quality of heuristic solutions.
We address this challenge by introducing a novel method to validate heuristic MSSC solutions through optimality gaps. Our approach employs a divide-and-conquer strategy, decomposing the problem into smaller instances that can be handled by an exact solver. The decomposition is guided by an auxiliary optimization problem, the ``anticlustering problem", for which we design an efficient heuristic. Computational experiments demonstrate the effectiveness of the method for large-scale instances, achieving optimality gaps below 3\% while maintaining reasonable computational times. These results highlight the practicality of our approach in assessing feasible clustering solutions for large datasets, bridging a critical gap in MSSC evaluation.
\end{abstract}



\end{frontmatter}

\newpage

\section{Introduction}
Clustering is a powerful data analysis tool with applications across various domains. It addresses the problem of identifying homogeneous and well-separated subsets, called \emph{clusters}, from a given set \(O = \{p_1, \ldots, p_N\}\) of \(N\) data observations, each with \(D\) attributes, often referred to as the data dimension \citep{rao1971cluster, hartigan1979algorithm}. Homogeneity means that data within the same cluster should be similar, while separation implies that data observations in different clusters should be significantly different. Clustering is used in fields such as image segmentation, text mining, customer segmentation, and anomaly detection \citep{jain1999data}.
In many application domains, clustering is often used as a preprocessing step to support downstream optimization tasks \citep{croella2024}.
The most common type of clustering is partitioning, where we seek a partition \(\mathcal{P} = \{C_1, \ldots, C_K\}\) of \(O\) into \(K\) clusters such that:
\begin{enumerate}
    \item[(i)] \(C_j \neq \emptyset, \quad \forall j=1,\ldots,K\);
    \item[(ii)] \(C_i \cap C_j = \emptyset, \quad \forall i,j=1,\ldots,K\) and \(i \neq j\);
    \item[(iii)] \(\bigcup\limits_{j=1}^{K} C_j = O\).
\end{enumerate}
Clustering methods aim to optimize a given clustering criterion to find the best partition \citep{rao1971cluster}. For that purpose, they explore the exponential solution space of assigning data to clusters.
Let \({P}(O, K)\) denote the set of all \(K\)-partitions of \(O\). Clustering can be viewed as a mathematical optimization problem, where the clustering criterion \(f: {P}(O, K) \rightarrow \mathbb{R}\) defines the optimal solution for the clustering problem, expressed as:
\begin{equation} \label{clu_opt}
    \min \ \{f(\mathcal{P}): \mathcal{P} \in {P}(O, K)\}.
\end{equation}
The choice of the criterion \(f\) significantly affects the computational complexity of the clustering problem and is selected based on the specific data mining task. Among many criteria used in cluster analysis, the most intuitive and frequently adopted criterion is the minimum sum-of-squares clustering (MSSC) or $k$-means type clustering \citep{gambella2021optimization}, which is given by
\begin{equation}
    f(\mathcal{P})\ = \sum_{j=1}^K\sum_{i: p_i \in C_{j}} \| p_i - \mu_{j} \|^2,
\end{equation}
where \(\|\cdot\|\) is the Euclidean norm and \(\mu_j\) is the cluster center of the points \(p_i\) in cluster \(C_j\). 

The standard MSSC formulation is a Mixed-Integer Nonlinear Programming (MINLP) problem, which can be expressed as:
\begin{subequations}
\label{eq:MSSCcentroid}
\begin{align}
\textrm{MSSC}(O, K) = \min \ & \sum_{i=1}^N \sum_{j=1}^K \delta_{ij}\left\|p_i - \mu_j\right\|^2  \\
\textrm{s.\,t.} \ & \sum_{j=1}^K \delta_{ij} = 1, \quad \forall i \in [N],\label{eq:MSSCsingle}\\
& \sum_{i=1}^N \delta_{ij} \geq 1, \quad \forall j \in [K],\label{eq:MSSCempty}\\
& \delta_{ij} \in \{0, 1\}, \quad \forall i \in [N], \ \forall j \in [K],\\
& \mu_j \in \mathbb{R}^D, \quad \forall j \in [K],
\end{align}
\end{subequations}
where the notation $[Q]$ with any integer $Q$ denotes the set of indices $\{1,\ldots,Q\}$. The binary decision variable $\delta_{ij}$ expresses whether a data point $i$ is assigned to cluster $j$ ($\delta_{ij}=1)$ or not ($\delta_{ij}=0$). The objective function minimizes the sum of the squared Euclidean distances between the data points and the corresponding cluster centers $\mu_j$. Constraints ~\eqref{eq:MSSCsingle} make sure that each point is assigned to exactly one cluster and constraints ~\eqref{eq:MSSCempty} ensure that there are no empty clusters. Problem \eqref{eq:MSSCcentroid} is a nonconvex MINLP. 
Although not explicitly expressed in the formulation, the centers of the $K$ clusters  $\mu_j$, for $j\in [K]$, are located at the centroids of the clusters due to first-order optimality conditions on the $\mu_j$ variables.

The MSSC is known to be NP-hard in \(\mathbb{R}^2\) for general values of \(K\) \citep{mahajan2012planar} and in higher dimensions even for \(K=2\) \citep{aloise2009complexity}. The one-dimensional case is solvable in polynomial time, with an \(\mathcal{O}(KN^2)\) time and \(\mathcal{O}(KN)\) space dynamic programming algorithm \citep{wang2011ckmeans}. 
Achieving global optimality in MSSC is notoriously challenging in practice. The computational complexity grows exponentially with the number of data points, making exact solutions impractical for large-scale datasets. Consequently, many clustering algorithms resort to heuristic or approximate methods that offer no guarantees about the quality of the solution. However, clustering is a data mining task that requires global optimality. Unlike supervised methods, where class labels are known and many performance measures exist (e.g. precision, recall, F-score), clustering performance and validation are carried out \textit{internally} by optimizing a mathematical function as in~(\ref{clu_opt}) based only on the inter- and intra-cluster relations between the given data observations, or \emph{externally}, by measuring how close a given clustering solution is to a ground-truth solution for a specific data set. However, external validation is limited by the nonexistence of ground-truth solutions, which is why clustering is performed. Since clustering does not rely on side information, its outcome typically requires interpretation by domain experts. Poor-quality heuristic clustering can lead to incorrect interpretations and potentially serious consequences in fields such as medicine and finance.

\newpage

\paragraph{Contributions}
The impossibility of finding global solutions for large-scale instances poses significant challenges in applications where the quality of the clustering directly impacts downstream tasks, such as image segmentation, market segmentation, and bioinformatics. In particular, given the current state-of-the-art solvers, computing exact solutions for MSSC is out of reach beyond a certain instance size. Moreover, for large-scale datasets, even obtaining good lower bounds on the optimal MSSC objective value is prohibitive, making it impossible to assess the quality of heuristic solutions. This work is the first to fill this gap by introducing a method to validate heuristic clustering solutions.
Our approach is based on a divide-and-conquer strategy that decomposes the problem into small- and medium-sized instances that can be handled by exact solvers as computational engines.  
Importantly, our method does not require the optimal clustering as input: it can validate any feasible solution by computing a lower bound and the corresponding optimality gap, making it applicable to large-scale settings where exact optimization is infeasible.

Overall, we provide the following contributions:
\begin{itemize}

    \item We propose a scalable procedure to compute strong lower bounds on the optimal MSSC objective value. 
    These bounds allow the computation of certified optimality gaps for any feasible solution and provide a quantitative measure of solution quality, which can be used as a stopping criterion in heuristic frameworks. The approach relies on decomposing the original problem into smaller, tractable subproblems guided by an auxiliary ``anticlustering'' problem.

    \item We present a theoretical analysis of the quality of the lower bound, clarifying the conditions under which it is tight, and discussing the practical implications for algorithm design.
    
    \item We design an efficient heuristic for the  ``anticlustering problem'', and use it to validate heuristic MSSC solutions through certified optimality gaps.
    
    \item We demonstrate through extensive experiments that the method provides small optimality gaps for large-scale instances. The computed gaps are below 3\% in all but one case (around 4\%), with many instances showing much smaller gaps, while keeping computational times reasonable (below two hours on average).
    
\end{itemize}

The rest of this paper is organized as follows. Section \ref{sec:lit} reviews the literature related to the MSSC problem, describing heuristics and exact methods. Section \ref{sec:definition} provides a theoretical analysis of the quality of the lower bound and discusses its practical use. Section \ref{sec:avoc} illustrates the algorithm for validating a feasible MSSC solution. Section \ref{sec:results} provides computational results. Finally, Section \ref{sec:conclusions} concludes the paper and discusses research directions for future work.

\section{Literature review} \label{sec:lit}

The most popular heuristic to optimize MSSC is the famous $k$-means algorithm \citep{macqueen1967some, lloyd1982least} ($\approx$ 3M references in Google Scholar 2025). It starts from an initial partition of $O$, so that each data point is assigned to a cluster. In the sequel, the centroids are computed and each point is then assigned to the cluster whose centroid is closest to it. If there are no changes in assignments, the heuristic stops. Otherwise, the
centroids are updated, and the process is repeated. The main drawback of $k$-means is that it often converges to locally optimal solutions that may be far from the global minimum. Moreover, it is highly sensitive to the choice of the initial cluster centers. As a result, considerable research has been devoted to improving initialization strategies (see, e.g., \cite{arthur2006k,improvedkmeans2018,franti2019much}).

Several heuristics and metaheuristics have been developed for the MSSC problem. These include simulated annealing \citep{lee2021simulated}, nonsmooth optimization \citep{bagirov2006new, BAGIROV201612}, tabu search \citep{ALSULTAN19951443}, variable neighborhood search \citep{HANSEN2001405,Orlov2018, carrizosa2013variable}, iterated local search \citep{likas2003global}, evolutionary algorithms \citep{MAULIK20001455,SARKAR1997975}, difference of convex functions programming \citep{tao2014new,BAGIROV201612,KARMITSA2017367,KARMITSA2018245}. The \(k\)-means algorithm is also used as a local search subroutine in various algorithms, such as the population-based metaheuristic proposed by \cite{gribel2019hg}.

The methods for solving the MSSC problem to global optimality are generally based on branch-and-bound (B\&B) and column generation (CG) algorithms. The earliest B\&B for the MSSC problem is attributed to \cite{koontz1975branch}, later refined by \cite{diehr1985evaluation}. This method focuses on partial solutions with fixed assignments for a subset of the data points, exploiting the observation that the optimal solution value for the full set is no less than the sum of the optimal solutions for the subsets. \cite{brusco2006repetitive} further developed this approach into the repetitive-branch-and-bound algorithm (RBBA), which solves sequential subproblems with increasing data points. RBBA is effective for synthetic datasets up to 240 points and real-world instances up to 60 data points. Despite being classified as a B\&B algorithm, Brusco's approach does not leverage lower bounds derived from relaxations of the MINLP model. Instead, it relies on bounds derived from the inherent properties of the MSSC problem. In contrast, a more traditional line of research employs B\&B algorithms where lower bounds are obtained through appropriate mathematical programming relaxations of the MINLP model. For example, the B\&B method by \cite{sherali2005global} employs the reformulation-linearization technique to derive lower bounds by transforming the nonlinear problem into a 0-1 mixed-integer program. This method claims to handle problems up to 1000 entities, though replication attempts have shown high computing times for real datasets with around 20 objects \citep{aloise2011evaluating}. More recent efforts by \cite{burgard2023mixed} have focused on mixed-integer programming techniques to improve solver performance, but these have not yet matched the leading exact methods for MSSC.

The first CG algorithm for MSSC problem has been proposed by \cite{du1999interior}. In this algorithm, the restricted master problem is solved using an interior point method while for the auxiliary problem a hyperbolic program with binary variables is used to find a column with negative reduced cost. To accelerate the resolution of the auxiliary problem, variable-neighborhood-search heuristics are used to obtain a good initial solution. Although this approach successfully solved medium-sized benchmark instances, including the popular Iris dataset with 150 entities, the resolution of the auxiliary problem proved to be a bottleneck due to the unconstrained quadratic 0-1 optimization problem. To overcome this issue, \cite{aloise2012improved} proposed an improved CG algorithm that uses a geometric-based approach to solve the auxiliary problem. Specifically, their approach involves solving a certain number of convex quadratic problems to obtain the solution of the auxiliary problem. When the points to be clustered are in the plane, the maximum number of convex problems to solve is polynomially bounded. Otherwise, the algorithm needs to find the cliques in a certain graph induced by the current solution of the master problem to solve the auxiliary problems. The efficiency of this algorithm depends on the sparsity of the graph, which increases as the number of clusters $K$ increases. Therefore, the algorithm proposed by \cite{aloise2012improved} is particularly efficient when the graph is sparse and $K$ is large. Their method was able to provide exact solutions for large-scale problems, including one instance of 2300 entities, but only when the ratio between $N$ and $K$ is small. More recently, \cite{sudoso2024column} combined the CG proposed by \cite{aloise2012improved} with dynamic constraint aggregation \citep{bouarab2017dynamic} to accelerate the resolution of the restricted master problem, which is known to suffer from high degeneracy.  When solving MSSC instances in the plane using branch-and-price, the authors show that this method can handle datasets up to about 6000 data points.

Finally, there is a large branch of literature towards the application of techniques from semidefinite programming (SDP). \cite{peng2005new} and \cite{peng2007approximating} showed the equivalence between the MINLP model and a 0-1 SDP reformulation. \cite{aloise2009branch} developed a branch-and-cut algorithm based on the linear programming (LP) relaxation of the 0-1 SDP model, solving instances up to 202 data points. \cite{piccialli2022sos} further advanced this with a branch-and-cut algorithm using SDP relaxation and polyhedral cuts, capable of solving real-world instances up to 4,000 data points. This algorithm currently represents the state-of-the-art exact solver for MSSC. Additionally, \cite{liberti2022side} examine the application of MINLP techniques to MSSC with side constraints, while the SDP-based exact algorithms proposed in \cite{piccialli2022semi, piccialli2023global} demonstrate state-of-the-art performance for constrained variants of MSSC.

\section{Decomposition strategy and lower bound computation} \label{sec:definition}
By exploiting the Huygens' theorem (e.g., \cite{edwards1965method}), stating that the sum of squared Euclidean distances from individual points to their cluster centroid is equal to the sum of squared pairwise Euclidean distances, divided by the number of points in the cluster, Problem \eqref{eq:MSSCcentroid} can be reformulated as 
\begin{subequations}
\label{eq:MSSC_intdata}
\begin{align}
\textrm{MSSC}(O,K) = \min_{} & \sum_{j=1}^K\frac{\displaystyle\sum_{i=1}^N\sum_{i^\prime=1}^N \delta_{ij}\delta_{i^\prime j}||p_i - p_{i'}||^2 }{2 \, \cdot \displaystyle\sum_{i=1}^N\delta_{ij}}\\
\text{s.t.}~ & \sum_{j=1}^K \delta_{ij} = 1,\quad \forall i \in [N],\\
& \sum_{i=1}^N \delta_{ij}  \ge 1, \quad \forall j \in [K],\\
& \delta_{ij} \in \{0,1\}, \quad \forall i \in [N], \ \forall j \in [K].
\end{align}
\end{subequations}


Our idea is to exploit the following fundamental result provided in \cite{koontz1975branch,diehr1985evaluation} and also used in \cite{brusco2006repetitive}.

\begin{proposition}\label{prop1}\citep{koontz1975branch, diehr1985evaluation}
Assume the dataset $O=\{p_1,\ldots,p_N\}$ is divided into $T$ subsets $\mathcal{S} = \{S_1,\ldots,S_T\}$ such that $\cup_{t=1}^TS_t = O$, $S_t\cap S_{t^\prime}=\emptyset$ and $S_t \neq \emptyset$ for all $t,t^\prime\in [T]$, $t\not=t^\prime$, i.e., $\mathcal{S}$ is a partition of O. 
Assume also that the optimal value of the MSSC problem on each subset is available, i.e., $\textrm{MSSC}(S_t,K)$ is known, for all $t\in [T]$. Then:
\begin{equation}\label{eq:lower}
   \textrm{MSSC}\left(O,K\right)=\displaystyle \textrm{MSSC}\left(\displaystyle\bigcup_{t=1}^T S_t,K\right)   \geq \displaystyle\sum_{t=1}^T \textrm{MSSC}(S_t,K) = LB^*.
\end{equation}
\end{proposition}
Proposition \ref{prop1} states that the sum of the optimal values on all the subsets provides a lower bound $LB^*$ on the optimal value of Problem \eqref{eq:MSSCcentroid}. 

To our knowledge, a theoretical analysis of the quality of the lower bound in  Eq. \eqref{eq:lower} has not been performed in the literature. In particular, the following research questions naturally arise: when and why is the bound tight? Given a number $T$ of subsets, what characteristics of the partition influence the quality of the bound? The next subsections are devoted to answering these questions.

\subsection{Theoretical Analysis}\label{sec:theoretical}
The first step of the analysis focuses on how the MSSC objective decomposes when the dataset is partitioned. For a cluster $C_j$ with centroid $\mu_j$, $j \in [K]$, we can state the following lemma.

\newpage

\begin{lemma}\label{lem:anova}
Let $\mathcal{P}=\{C_1,\dots,C_K\}$ be any $K$-partition of $O$ with cluster centroids
$\mu_j:=\frac{1}{|C_j|}\sum_{i:p_i\in C_j}p_i$. Let $\mathcal{S}=\{S_1,\dots,S_T\}$ be any partition of $O$.
For a fixed $j\in[K]$ and $t\in[T]$, define the sets $A_{jt}:=C_j\cap S_t$ (i.e., the set of points of cluster 
$C_j$ that lie in subset $S_t$), their sizes
$n_{jt}:=|A_{jt}|$, and their means $\mu_{jt}:=\frac{1}{n_{jt}}\sum_{i:p_i\in A_{jt}} p_i$ for $n_{jt}>0$. Then
\begin{equation}\label{eq:anova}
\sum_{i:p_i\in C_j}\|p_i-\mu_j\|^2
\;=\;
\sum_{t=1}^T \sum_{i:p_i\in A_{jt}}\|p_i-\mu_{jt}\|^2
\;+\;\sum_{t=1}^T n_{jt}\,\|\mu_{jt}-\mu_j\|^2 .
\end{equation}
\end{lemma}
\begin{proof}
Consider a cluster $C_j$ and a subset $S_t$. For each $p_i\in A_{jt}$, let 
\[
p_i-\mu_j \;=\; (p_i-\mu_{jt}) + (\mu_{jt}-\mu_j),
\]
and expand the squared norm:
\begin{align*}
\|p_i-\mu_j\|^2
= \|(p_i-\mu_{jt}) + (\mu_{jt}-\mu_j)\|^2 = \|p_i-\mu_{jt}\|^2 \;+\; 2\langle p_i-\mu_{jt},\,\mu_{jt}-\mu_j\rangle \;+\; \|\mu_{jt}-\mu_j\|^2.
\end{align*}
Now sum this identity over all points $p_i\in A_{jt}$:
{\small
\begin{align}\label{eq:objdec}
\sum_{i:p_i\in A_{jt}}\|p_i-\mu_j\|^2
&= \sum_{i:p_i\in A_{jt}}\|p_i-\mu_{jt}\|^2
\;+\; 2\left\langle \sum_{i:p_i\in A_{jt}}(p_i-\mu_{jt}),\,\mu_{jt}-\mu_j\right\rangle
\;+\; \sum_{i:p_i\in A_{jt}}\|\mu_{jt}-\mu_j\|^2.
\end{align}}
The middle “cross” term in Eq. \eqref{eq:objdec} vanishes since $\mu_{jt}$ is the mean of the points in $A_{jt}$:
\[
\sum_{i:p_i\in A_{jt}}(p_i-\mu_{jt}) \;=\; \left(\sum_{i:p_i\in A_{jt}} p_i\right) - n_{jt}\,\mu_{jt} \;=\; 0.
\]
The last term in Eq. \eqref{eq:objdec} does not depend on $p_i$, hence summed up $n_{jt}$ times, equals
$n_{jt}\,\|\mu_{jt}-\mu_j\|^2$. Therefore, for each $t\in[T]$ we have
\[
\sum_{i:p_i\in A_{jt}}\|p_i-\mu_j\|^2
\;=\; \sum_{i:p_i\in A_{jt}}\|p_i-\mu_{jt}\|^2 \;+\; n_{jt}\,\|\mu_{jt}-\mu_j\|^2.
\]
Finally, summing the previous equality over $t \in [T]$ yields Eq. \eqref{eq:anova}.
\end{proof}

We now introduce the definition of the restriction of a cluster assignment to a subset $S_t$.

Given a $K$-partition $\mathcal{P}=\{C_1,\dots,C_K\}$ of $O$ and a subset $S_t\subseteq O$,
the \emph{restriction} of $\mathcal{P}$ to $S_t$ is
\[
\mathcal{P}|_{S_t}\;:=\;\{C_1\cap S_t,\;C_2\cap S_t,\;\dots,\;C_K\cap S_t\}.
\]
The restriction $\mathcal{P}|_{S_t}$ is a $K$-partition of $S_t$ obtained by selecting from each cluster $C_j$ the data points in $S_t$. We denote by $f(\mathcal{P}|_{S_t})$ the MSSC value on $S_t$ obtained by keeping the cluster assignment associated with $\mathcal{P}$ and recomputing the centroids on the restricted subsets $C_j\cap S_t$. Note that empty intersections contribute $0$ to $f(\mathcal{P}|_{S_t})$. We now give the following result for the restriction of a cluster assignment.
\begin{lemma}\label{lem:restriction-baseline}
For any $K$-partition $\mathcal{P}$ of $O$ and any subset $S_t\subseteq O$,
\[
\textrm{MSSC}(S_t,K)\;\le\; f(\mathcal{P}|_{S_t}),
\]
and equality holds if and only if $\mathcal{P}|_{S_t}$ is an optimal $K$-partition for $S_t$.
\end{lemma}
\begin{proof}
By definition, $\textrm{MSSC}(S_t,K)=\min\{\,f(Q): Q \text{ is a $K$-partition of }S_t\,\}$. Since $Q=\mathcal{P}|_{S_t}$ is a $K$-partition of $S_t$, the inequality is verified. The equality holds when the minimum is attained in $\mathcal{P}|_{S_t}$.
\end{proof}

\noindent
Based on the previous analysis, the next proposition gives an analytical description of the gap between the optimal MSSC value and the lower bound $LB^*$ in Proposition \ref{prop1}.

\begin{proposition}\label{prop:tight}
Let $\mathcal{P}^*=\{C_1^*,\dots,C_K^*\}$ with centroids $\{\mu^*_1,\ldots,\mu^*_K\}$ be an optimal MSSC solution on $O$ with value $\textrm{MSSC}(O,K)$,
and let $\mathcal{S}=\{S_1,\dots,S_T\}$ be any partition of $O$. Denote by $\textrm{MSSC}(S_t,K)$ the optimal MSSC value on $S_t$, and let $\mu_{jt}=\frac{1}{n_{jt}}\sum_{i:p_i\in C_{j}^*\cap S_t} p_i$ where $n_{jt}=|C_{j}^*\cap S_t| > 0$. Then,
{\small
\begin{align}
\textrm{MSSC}(O,K) -\sum_{t=1}^T \textrm{MSSC}(S_t,K)
&=\underbrace{\sum_{t=1}^T\Big[f\!\left(\mathcal{P}^*|_{S_t}\right)-\textrm{MSSC}(S_t,K)\Big]}_{\text{(A) per-subset suboptimality }}
+
\underbrace{\sum_{j=1}^K\sum_{t=1}^T n_{jt}\,\|\mu_{jt}-\mu^*_j\|^2}_{\text{(B) between-subset dispersion}} .
\label{eq:gap-master}
\end{align}}
\end{proposition}
\begin{proof}
By Lemma~\ref{lem:anova}, for each cluster $C_j^*$ we have
\begin{align}\label{eq:9}
\sum_{i:p_i\in C_j^*}\|p_i-\mu_j^*\|^2
=\sum_{t=1}^T \sum_{i:p_i\in C_j^*\cap S_t}\|p_i-\mu_{jt}\|^2 + \sum_{t=1}^T n_{jt}\,\|\mu_{jt}-\mu^*_j\|^2.
\end{align}
Summing Eq. \eqref{eq:9} for all $j \in [K]$ gives
\[
\textrm{MSSC}(O,K)=\sum_{t=1}^T f\!\left(\mathcal{P}^*|_{S_t}\right) \;+\; \sum_{j=1}^K \sum_{t=1}^T n_{jt}\,\|\mu_{jt}-\mu^*_j\|^2.
\]
Then, subtracting $\sum_{t=1}^T \textrm{MSSC}(S_t,K)$ on both sides of the equation above we obtain Eq. \eqref{eq:gap-master}. 
\end{proof}

Note that the right-hand side of Eq. \eqref{eq:gap-master} is always nonnegative, and in general, there will be a gap. This gap is generated by two nonnegative terms: the per-subset suboptimality (A) and the between-subset dispersion (B). Now we investigate under which conditions these two terms vanish.

\newpage

\begin{assumption}\label{ass0}[Restriction optimality] Let $\mathcal{P}^*=\{C^*_1,\dots,C^*_K\}$ be an optimal MSSC solution on $O$ with value $\textrm{MSSC}(O,K)$,
and let $\mathcal{S}=\{S_1,\dots,S_T\}$ be any partition of $O$. Denote by $\textrm{MSSC}(S_t,K)$ the optimal MSSC value on $S_t$. For each $t \in [T]$, the restriction $\mathcal{P}^*|_{S_t}$
attains $\textrm{MSSC}(S_t,K)$; i.e., $f(\mathcal{P}^*|_{S_t})=\textrm{MSSC}(S_t,K)$.
\end{assumption}

Under Assumption \ref{ass0}, Proposition \ref{prop:tight} implies
\begin{align*}
\textrm{MSSC}(O,K) -\sum_{t=1}^T \textrm{MSSC}(S_t,K)
&=\sum_{j=1}^K\sum_{t=1}^T n_{jt}\,\|\mu_{jt}-\mu^*_j\|^2.
\end{align*}
i.e., the gap only depends on the distance between the centroids of the optimal clustering and the centroids of the restrictions. We can finally state the equality/tightness condition for the lower bound in Proposition \ref{prop1}, which is given in the following proposition.

\begin{proposition}\label{prop:tightass2}
Let $\mathcal{P}^*=\{C^*_1,\dots,C^*_K\}$ be an optimal MSSC solution on $O$ with value $\textrm{MSSC}(O,K)$,
and let $\mathcal{S}=\{S_1,\dots,S_T\}$ be a partition of $O$. Furthermore, let $\mathcal{P}_t^*=\{C^*_{1t},\dots,C^*_{Kt}\}$ be an optimal MSSC solution on $S_t$ with centroids $\{\mu^*_{1t},\ldots,\mu^*_{Kt}\}$ and value $\textrm{MSSC}(S_t,K)$. The equality
\[
\textrm{MSSC}(O,K)=\sum_{t=1}^T \textrm{MSSC}(S_t,K)
\]
holds if and only if $\mu^*_{jt}=\mu^*_j$ for every $j\in[K]$ and $t\in[T]$.
\end{proposition}
\begin{proof}
Using Lemma \ref{lem:restriction-baseline}, $f(\mathcal{P}^*|_{S_t})-\textrm{MSSC}(S_t,K)\ge 0$ for each $t \in [T]$ by optimality of $\textrm{MSSC}(S_t,K)$. Moreover, the between-subset term (B) is nonnegative by definition. If equality holds, the two nonnegative terms (A) and (B) in Eq. \eqref{eq:gap-master} must be zero. When term (A) is zero, the restriction $\mathcal{P}^*|_{S_t}$ is optimal, meaning that $\mathcal{P}^*_t=\mathcal{P}^*|_{S_t}$ so that the optimal centroids are $\mu_{jt}=\mu_{jt}^*$ for $t\in [T]$ and $j\in [K]$. Term (B) being zero implies $\mu_{jt}=\mu^*_j$ for every $j\in[K]$ and $t\in[T]$.\par
To prove the other direction, assume $\mu^*_{jt}=\mu^*_j$ for all $j\in[K]$ and $t\in[T]$. Since the optimal centroids $\mu_{ji}^*$ of the clustering on $S_t$ for $t\in [T]$ are $T$ replicas of the centroids $\mu_j^*$ of the optimal clustering on the original dataset $O$, the optimal clustering in each subset $S_t$ will coincide with the restriction of the optimal clustering, i.e., $f(\mathcal{P}^*|_{S_t})=\textrm{MSSC}(S_t,K)$, vanishing the term (A). Furthermore, term (B) vanishes since $\mu^*_{jt}=\mu^*_j$ for every $j\in[K]$ and $t\in[T]$.
 \end{proof}

We illustrate the tightness condition in Proposition \ref{prop:tightass2} with an example. 
Figure \ref{fig:exdata} illustrates the optimal clustering partition for a synthetic dataset with $N=64$ points, $D=2$ features, and $K=4$ clusters. 
The points are arranged in four well-separated groups, with centroids $\mu^*_j$ for every cluster $j \in \{1,2,3,4\}$ and optimal solution value $\text{MSSC}({O},K) = 575.674$. 
This clustering satisfies Assumption \ref{ass0}, and represents the reference solution used in our example. 
Figure \ref{fig:example2} reports a partition $\mathcal{S}$ where the subfigures correspond to the four subsets $S_t$ $(t=1,\dots,4)$. We display the points assigned to $S_t$, the solution of the clustering problem restricted to $S_t$, and the corresponding centroids $\mu_{jt}$ for each cluster $j$. It can be seen that partition $\mathcal{S}$ satisfies the hypothesis of Proposition \ref{prop:tightass2}. 
The resulting lower bound is $\text{MSSC}(O,K) =\sum_{t=1}^T \text{MSSC}(S_t,K)$, which yields an optimality gap of $0\%$, thereby certifying the tightness of the lower bound on $\mathcal{S}$.

\begin{figure}[!ht]
    \centering
    \begin{subfigure}{0.4\textwidth}
        \includegraphics[scale=0.5]{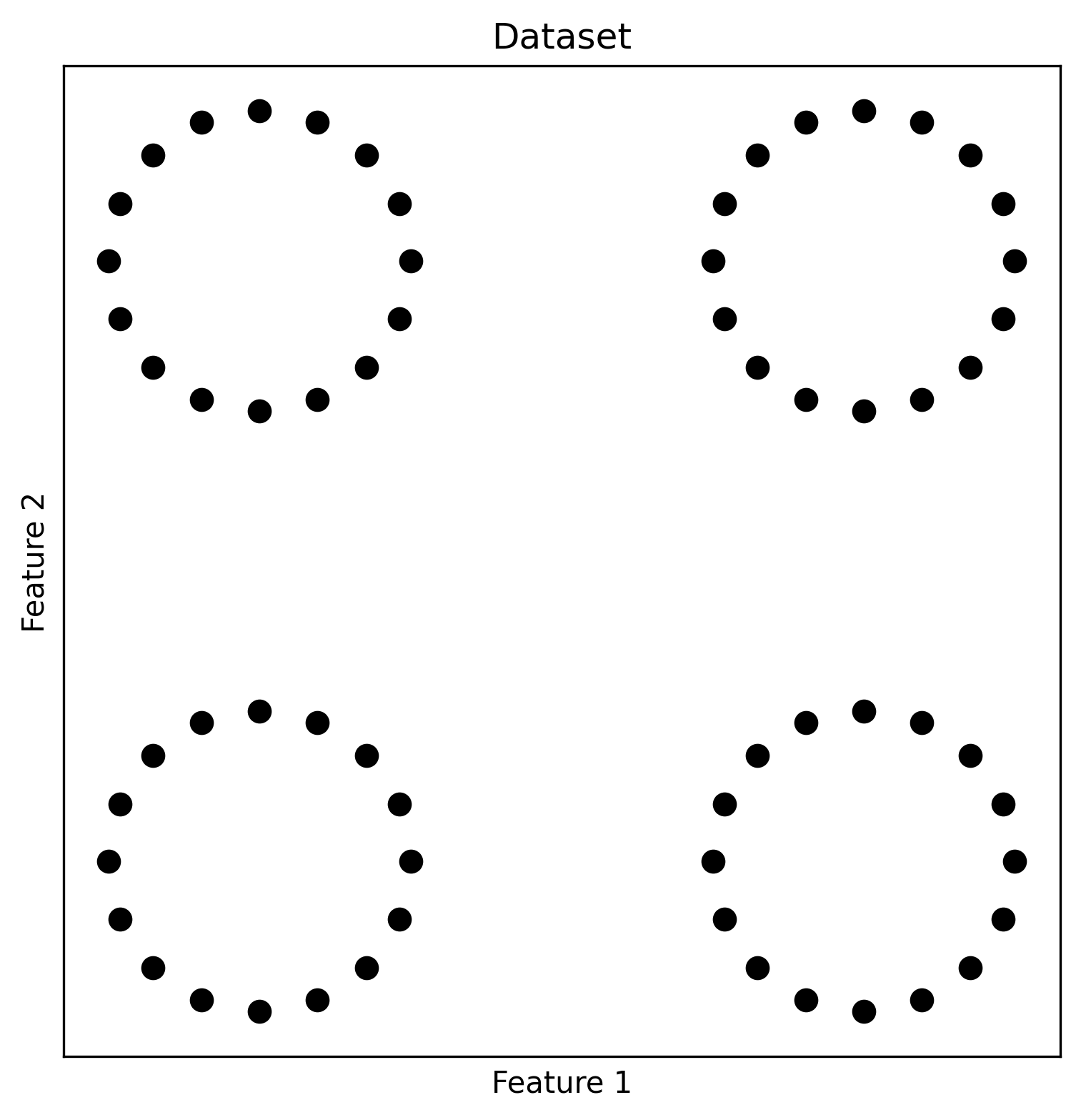}
    \end{subfigure}    
    \hfill
    \begin{subfigure}{0.4\textwidth}
        \includegraphics[scale=0.5]{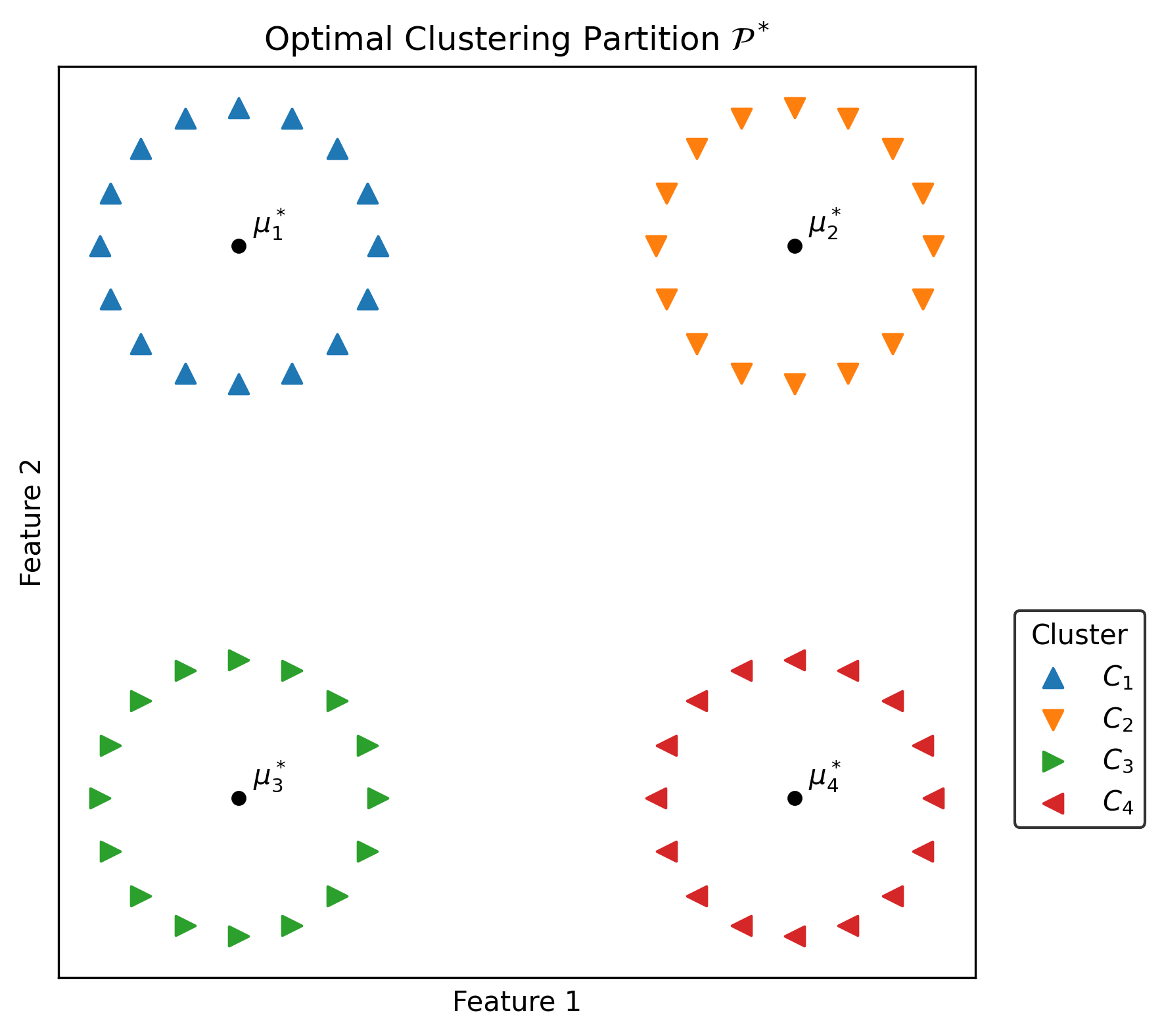}
    \end{subfigure}
    \caption{A synthetic dataset of $N=64$ points, $D=2$ features with $K=4$ natural well-separated clusters {(on the left)} and its optimal clustering partition {(on the right)}.}
    \label{fig:exdata}
\end{figure}

\begin{figure}[!ht]
    \centering
    \begin{subfigure}{0.4\textwidth}
        \caption{MSSC$(S_1, 4)= 143.9185$}
        \includegraphics[scale=0.5]{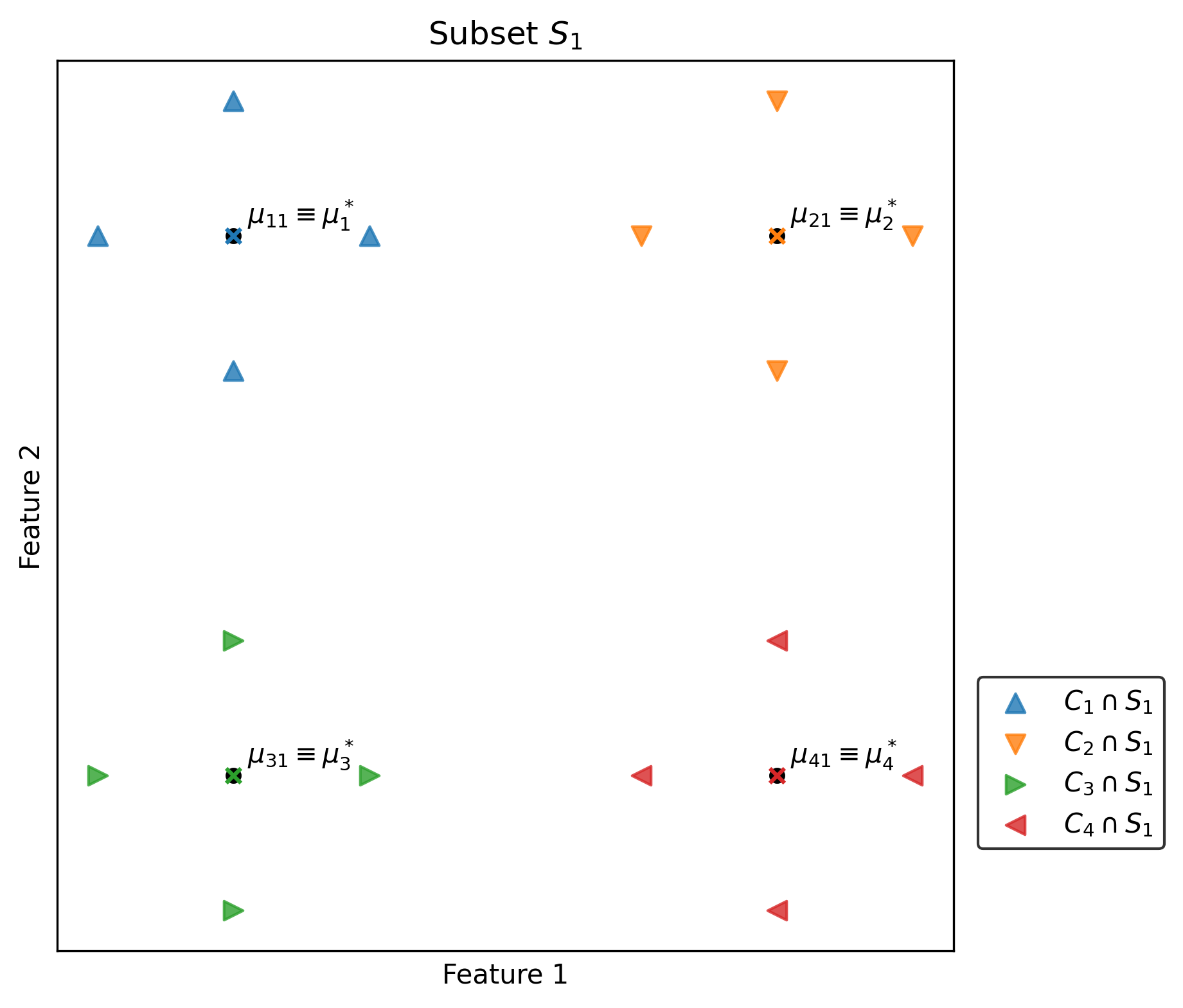}
        \label{fig:sol1}
    \end{subfigure}
    \hfill
    \begin{subfigure}{0.4\textwidth}
        \caption{MSSC$(S_2, 4)= 143.9185$}
        \includegraphics[scale=0.5]{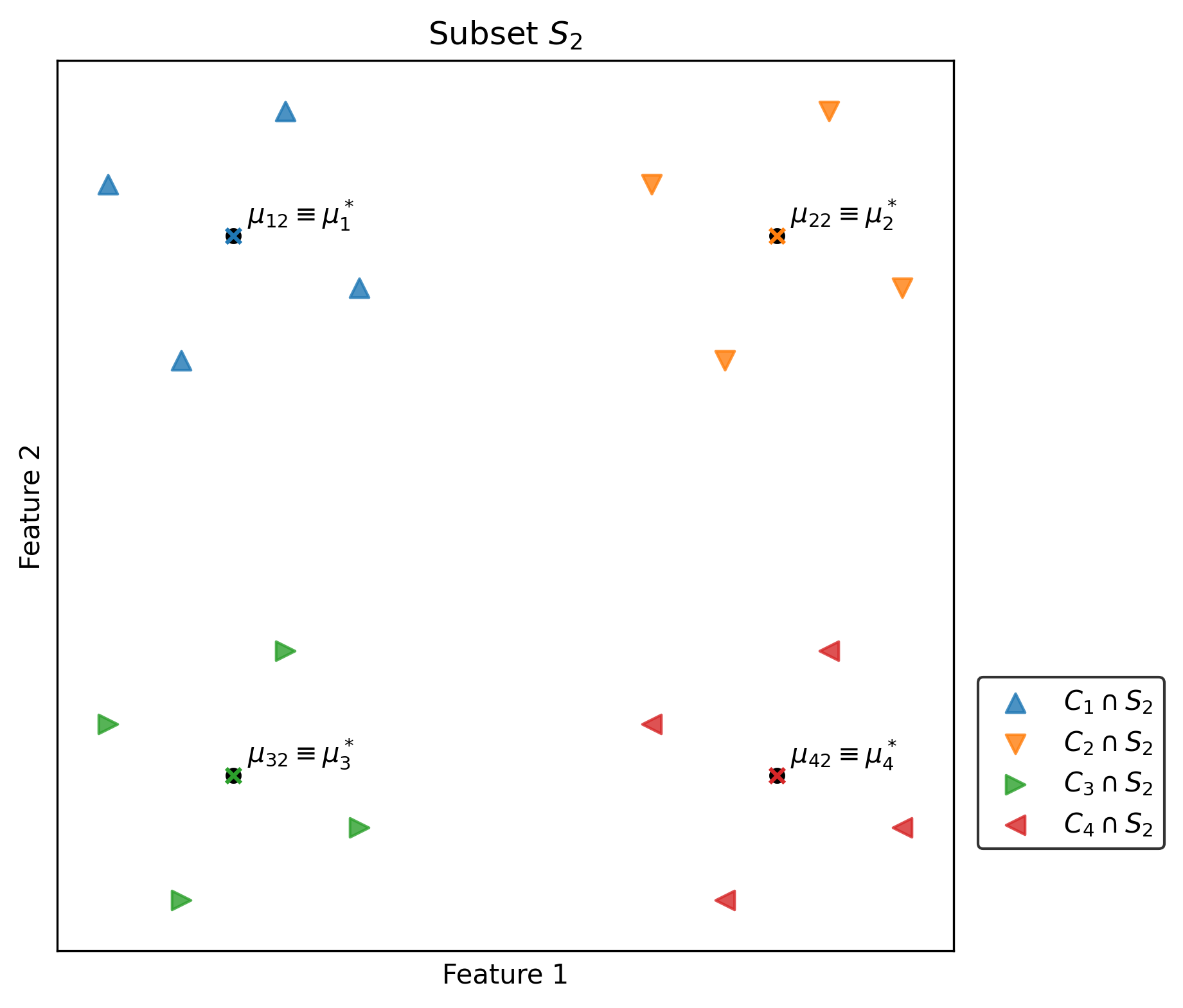}
        \label{fig:sol2}
    \end{subfigure}\\
    \begin{subfigure}{0.4\textwidth}
        \centering
        \caption{MSSC$(S_3, 4)= 143.9185$}
        \includegraphics[scale=0.5]{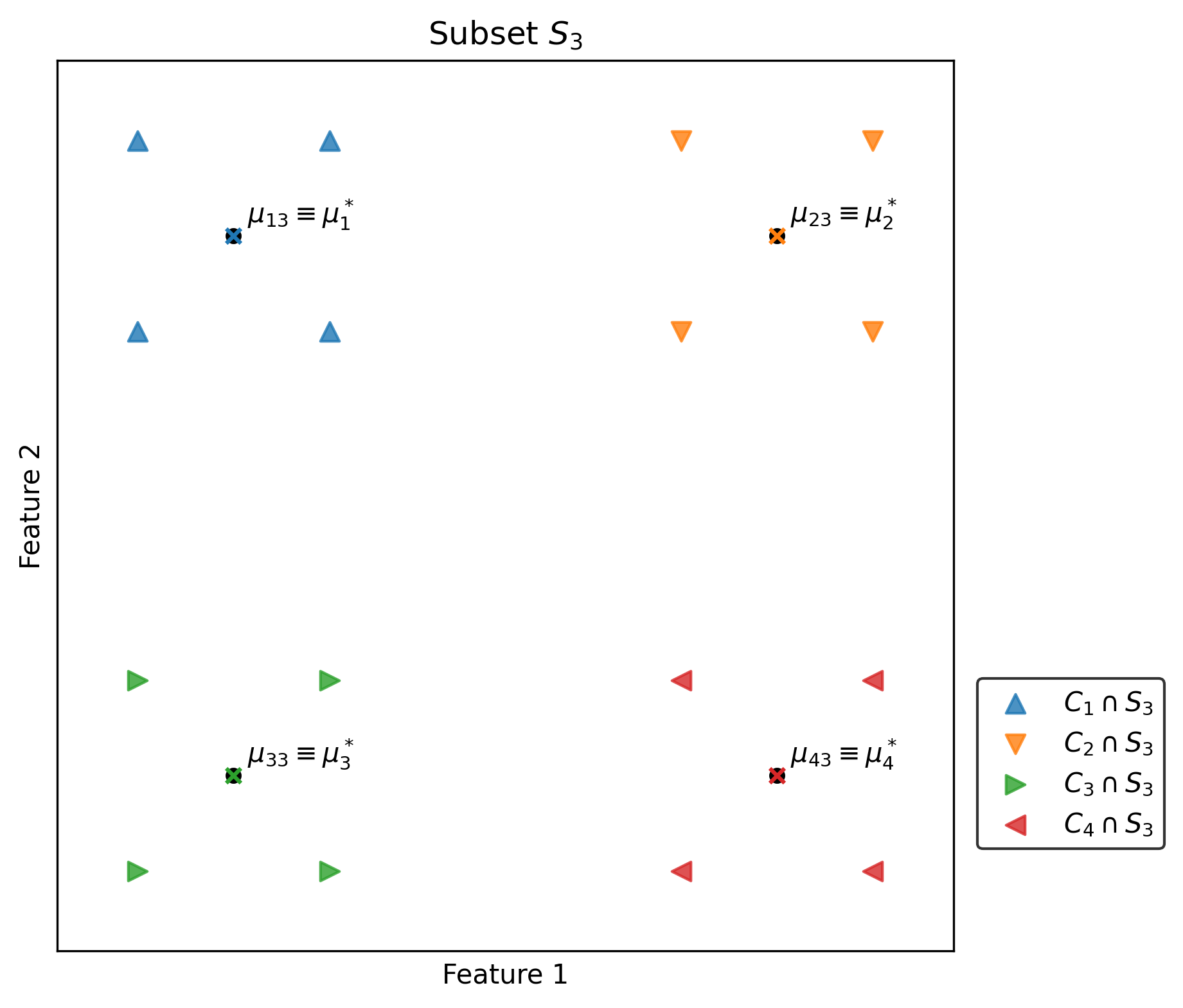}
        \label{fig:sol3}
    \end{subfigure}    
    \hfill
    \begin{subfigure}{0.4\textwidth}
        \centering
        \caption{MSSC$(S_4, 4)= 143.9185$}
        \includegraphics[scale=0.5]{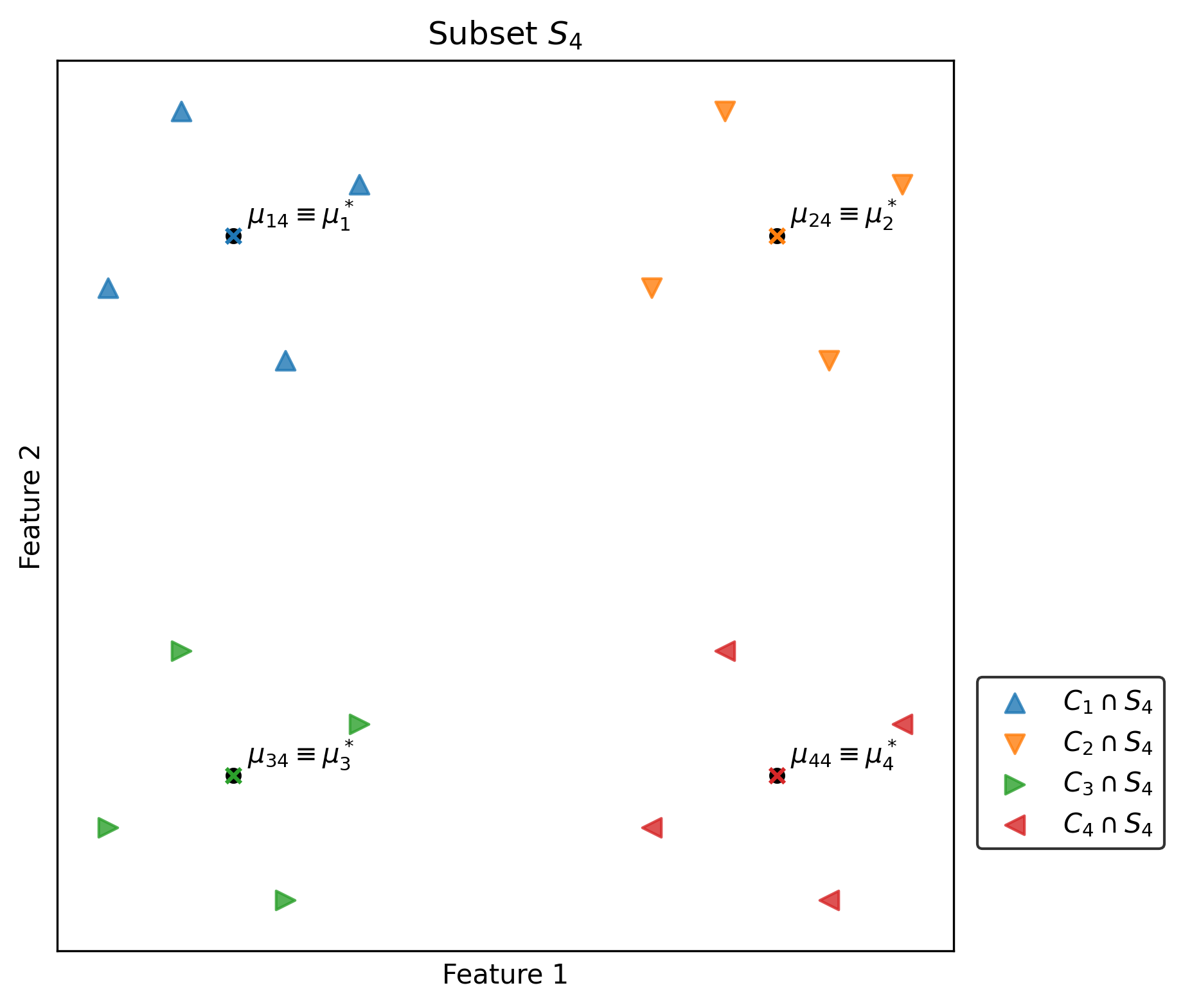}
        \label{fig:sol4}
    \end{subfigure}
    \caption{ A Partition $\mathcal{S}$ of the dataset of Figure \ref{fig:exdata} in $T=4$ subsets and corresponding optimal clustering on each subset. Here, $\textrm{MSSC}(O, K) =  575.674=\sum_{t=1}^T$MSSC$(S_t, K)$ and hence the lower bound is tight.}
    \label{fig:example2}
\end{figure}

\subsection{Practical implications}

Our idea is to produce a quality certificate for a heuristic solution whose objective value yields an upper bound on Problem \eqref{eq:MSSCcentroid}.  To this end, exploiting Proposition \ref{prop1} and the theoretical analysis described in Section \ref{sec:theoretical}, we build a partition $\mathcal{S}$ of the dataset $O$ that can be used to compute a tight lower bound on $\textrm{MSSC}(O,K)$. Computing such a lower bound requires to solve to optimality the MSSC subproblems arising in the subsets of the considered partition. However, solving to global optimality the subproblems, even though tractable with state-of-the-art solvers, can be computationally demanding. We can therefore exploit the following corollary, which is an immediate consequence of Proposition \ref{prop1}.

\begin{corollary} \label{cor}Assume the dataset $O=\{p_1,\ldots,p_N\}$ is partitioned into $T$ subsets $\mathcal{S} = \{S_1,\ldots,S_T\}$. Then
\begin{equation}\label{eq:lower_weak}
   \textrm{MSSC}_\textrm{UB}(O,K) \geq \textrm{MSSC}(O,K)\geq \displaystyle\sum_{t=1}^T  \textrm{MSSC}_\textrm{LB}(S_t,K) = LB
\end{equation}
holds for any upper bound $\textrm{MSSC}_\textrm{UB}(S_t,K)$ and lower bound $\textrm{MSSC}_\textrm{LB}(S_t,K)$ on $\textrm{MSSC}(S_t,K)$.    
\end{corollary}

Clearly, $LB^*$ in Proposition \ref{prop1} is greater than or equal to $LB$. Moreover, while the quality of $LB^*$ depends only on the partition ${\cal S}$, the quality of $LB$ depends both on ${\cal S}$ and on the strength of the lower bounds $\textrm{MSSC}_\textrm{LB}(S_t,K)$ computed for the subproblems. The recently introduced exact solver \textsc{SOS}-\textsc{SDP}, proposed by \cite{piccialli2022sos}, has significantly advanced the state of the art in solving the MSSC problem. It combines SDP relaxations with a branch-and-cut framework, producing strong and computationally efficient bounds. Empirical evidence shows that \textsc{SOS}-\textsc{SDP} can solve problems to proven optimality of substantially larger size than previously possible, handling instances with up to about 4,000 data points. At the same time, its performance at the root node of the branch-and-bound tree is particularly effective: for problems of up to roughly 1,000 points, the solver computes very tight bounds within a reasonable time, often with only a negligible optimality gap. This makes \textsc{SOS}-\textsc{SDP} the most suitable choice both for solving the subproblems exactly and for producing the lower bounds $\textrm{MSSC}_\textrm{LB}(S_t,K)$ in Corollary~\ref{cor}. In the following, we restrict each subset $S_t$ to about 1,000 data points, ensuring that the subproblems remain tractable while preserving the quality of the resulting lower bounds.

Now, we focus our efforts on the following research question: what is the partition $\mathcal{S} = \{S_1, \dots, S_T\}$ of $O$ into $T$ subsets that maximizes the lower bound obtained by summing up the MSSC objective computed on each subset? 

To define the partition $\mathcal{S}$, we need to choose the number of subsets $T$ and the cardinality of each subset $S_t$. Although determining the optimal clustering for each subset can be fast if the subsets are small, smaller subproblems often result in poor bounds (e.g., a subproblem with $K$ observations has an objective value of MSSC equal to zero). Anyway, the size of the subproblems should allow for efficient computation. Note that a very good lower bound can be found by producing a partition with a small number of sets $T$ where most of the data set is contained in one set, and the other sets contain only a very small number of points.
An extreme case where the lower bound in Eq. \eqref{eq:lower} provides almost zero optimality gap is if $T=2$, $S_1$ is composed of all points apart from $k$, and $S_2$ contains only the $k$ points closest to the centers of the optimal clusters. However, this partition is useless since computing a valid lower bound on $S_1$ is as difficult as on the original dataset. 

For practical efficiency, the computational effort spent by the solver on each subset $S_t$ should be comparable. This balance not only avoids bottlenecks but also enables efficient parallelization. To this end, we focus on the equipartition case, where the dataset is divided into subsets of (approximately) equal size, and this size is chosen to remain within the tractable regime. In this way, either the exact solution of the subproblem can be obtained, or the lower bound in Eq. \eqref{eq:lower_weak} can be computed within limited computational time. Consequently, we restrict our attention to partitions into subsets of equal size, with each subset containing fewer than 1,000 points. When the subsets are of comparable size, the lower bounds $\textrm{MSSC}_\textrm{LB}(S_t,K)$ can be computed efficiently across all subproblems. Moreover, if the partition is of good quality, each subset contributes meaningfully to the overall lower bound, avoiding cases where some subsets add only marginal improvement.
So, the research question now becomes: what is the partition $\mathcal{S} = \{S_1, \dots, S_T\}$ of the dataset $O$ into $T$ subsets of equal size maximizing the lower bound obtained by summing up the MSSC objective computed on each subset?

Formally, we would like to solve the following optimization problem:
\begin{subequations}
\label{eq:PART}
\begin{align}
\textrm{LB} = \max_{}~ & \sum_{t=1}^T \textrm{MSSC}(S_t,K) \\\label{eq:asspart}
\text{s.t.}~ & \sum_{t=1}^T \xi_{it} = 1,\quad \forall i \in [N],\\\label{eq:cardpart}
& \sum_{i=1}^N \xi_{it}  =  \frac{N}{T}, \quad \forall t \in [T],\\
& \xi_{it} \in \{0,1\}, \quad \forall i \in [N], \ \forall t \in [T],
\end{align}
\end{subequations}
where the binary variables $\xi_{it}$ are equal to 1 if the point $p_i$ is assigned to the subset $S_t$. Constraints \eqref{eq:asspart} require that each point is assigned to a single subset, constraints \eqref{eq:cardpart} require that each subset $S_t$ has a number of elements equal to $N/T$. Note that for simplicity, we assume that $N$ can be divided by $T$. The objective function is the solution of the following problem:
\begin{subequations}
\label{eq:MSSC}
\begin{align}
\textrm{MSSC}(S_t, K) = \min_{} & \sum_{j=1}^K \frac{\sum_{i=1}^{N} \sum_{i'=1}^N \delta^t_{ij} \delta^t_{i'j} ||p_i - p_{i'}||^2 }{2 \, \cdot \sum_{i=1}^N \delta^t_{ij}}\\\label{eq:connect}
\text{s.t.}~ & \sum_{j=1}^K \delta^t_{ij} = \xi_{it},\quad \forall i \in [N],\\
& \sum_{i=1}^N \delta^t_{ij}  \ge 1, \quad \forall j \in [K],\\
& \delta^t_{ij} \in \{0,1\}, \quad \forall i \in [N], \ \forall j \in [K]\; 
\end{align}
\end{subequations}
that is, the MSSC problem on each subset $S_t$. The two problems are linked by constraints \eqref{eq:connect}, which impose that the data points assigned to the subset $S_t$ must be assigned to the $K$ clusters.

However, Problem \eqref{eq:PART} is not directly solvable, due to its bilevel nature. This undesirable bilevel nature could be avoided if the optimal clustering for any subset were known in advance and independent of the chosen partition. In this case, the optimization variables of the second level would be known, and the problem would become a single-level one. For instance, this situation arises when Assumption~\ref{ass0} holds. In fact, Assumption~\ref{ass0} states that the restriction $\mathcal{P}^*|_{S_t}$ of the global optimal solution $\mathcal{P}^*$ remains optimal for each subset $S_t$, for all $t \in [T]$. However, in practice, we do not have access to $\mathcal{P}^*$ and instead rely on a heuristic solution $\mathcal{P}=\{C_1,\ldots,C_K\}$ with objective value $\textrm{MSSC}_{\textrm{UB}}(O,K)$. From this heuristic solution, we assign the lower-level variables $\delta^t_{ij}$ by restricting $\mathcal{P}$ to the subsets $S_t$, i.e., using $\mathcal{P}|_{S_t}$ for each $t \in [T]$. With this assignment, Problem~\eqref{eq:PART} becomes:

\begin{subequations}
\label{eq:lui}
\begin{align}
\textrm{LB} = \max_{} &  \sum_{t=1}^{T} \sum_{j=1}^K \sum_{i: p_i\in C_j} \sum_{i': p_{i'} \in C_j} \frac{\|p_i-p_{i'}\|^2}{2 \, \cdot} |C_j|/T \cdot \xi_{it} \cdot \xi_{i't}  \\
\text{s.t.}~ & \sum_{t=1}^T \xi_{it} = 1,\quad \forall i \in [N],\\
& \sum_{i:p_i\in C_j} \xi_{it} =  \frac{|C_j|}{T}, \quad \forall t \in [T], \ \forall j \in [K],\\
& \xi_{it} \in \{0,1\}, \quad \forall i \in [N], \ \forall t \in [T],
\end{align}
\end{subequations}
where, w.l.o.g., we assume that the number of points in each cluster $C_j$, for $j \in [K]$, is a multiple of $T$; that is, $|C_j|/T \in \mathbb{Z}$. Problem \eqref{eq:lui} can be decomposed in $K$ independent subproblems, one for each cluster $C_j$ with $j \in [K]$:
\begin{subequations}
\label{eq:cls}
\begin{align}
\max_{} &  \sum_{t=1}^{T} \sum_{i: p_i\in C_j} \sum_{i': p_{i'} \in C_j} \frac{\|p_i-p_{i'}\|^2}{2 \, \cdot |C_j|/T} \cdot \xi_{it} \cdot \xi_{i't}\\
\text{s.t.}~ & \sum_{t=1}^T \xi_{it} = 1,\quad \forall i:p_i \in C_j,\\
& \sum_{i: p_i \in C_j} \xi_{it} =  \frac{|C_j|}{T}, \quad \forall t \in [T],\\
& \xi_{it} \in \{0,1\}, \quad \forall i \in C_j\; \forall t \in [T].
\end{align}
\end{subequations}
Note that, the objective function of Problem~\eqref{eq:cls} is expressed in terms of pairwise distances. By applying Huygens’ theorem \citep{edwards1965method}, we can first reformulate it in terms of distances to centroids. Then, by applying Lemma~\ref{lem:anova}, the objective can be interpreted as the sum of squared distances between the centroids of the clustering solution on the full dataset $O$ and the corresponding centroids of the clustering solutions on the subsets $S_t$, for all $t \in [T]$. Hence, solving Problem~\eqref{eq:cls} is equivalent to minimizing the gap in Eq.~\eqref{eq:gap-master}.

Problem \eqref{eq:cls} is related to the so-called ``anticlustering problem'', first defined in \cite{spath1986anticlustering}. The goal is to partition a set of objects into groups with high intragroup dissimilarity and low intergroup dissimilarity, meaning that the objects within a group should be as diverse as possible, whereas pairs of groups should be very similar in composition.  In \cite{spath1986anticlustering}, the objective was to maximize rather than minimize the within-group sum-of-squares criterion, which translates into Problem \eqref{eq:cls}. 
The term ``anticlustering'' is used in psychological research, where many applications require a pool of elements to be partitioned into multiple parts while ensuring
that all parts are as similar as possible.  As an example, when dividing students into groups in a school setting, teachers are often interested in assembling groups that are
similar in terms of scholastic ability and demographic composition. In educational
psychology, it is sometimes necessary to divide a set of test items into parts of equal length that are presented to different cohorts of students. To ensure fairness of the test, it is required that all parts be equally difficult. 

The anticlustering problem is also considered in
\cite{papenberg2021using,brusco2020combining,papenberg2024k}, where different diversity criteria and heuristic algorithms are proposed. Problem \eqref{eq:cls} can also be related to the Maximally Diverse Grouping Problem (MDGP) \citep{LAI2016780,schulz2023balanced,schulz2023balanceda,SCHULZ202142}, where the set of vertices of a graph has to be partitioned into disjoint subsets of sizes within a given interval, while maximizing the sum of the edge weights in the subset. In this context, the weights represent the distances between the nodes. MDGP belongs to the category of vertex-capacitated clustering problems
\citep{johnson1993min}.
From now on, we call \textit{anticlusters} the subsets $S_t$ for $t\in[T]$.

Note that even if Assumption \ref{ass0} holds, the anticlustering partition can significantly influence the quality of the lower bound.
Consider again the dataset in Figure \ref{fig:exdata}. We know from Figure \ref{fig:example2} that we can get a zero optimality gap by choosing the right partition $\mathcal{S}$. However, if we choose the wrong partition we can fail in certifying the optimality and get a large gap.
Indeed, Figure~\ref{fig:example} shows an alternative anticlustering partition. Each subfigure corresponds to one of the four anticlusters $S_t$ $(t=1,\dots,4)$ and displays the points assigned to $S_t$, the clustering solution restricted to $S_t$, and the corresponding centroids $\mu_{jt}$ for each cluster $j$ $(j=1,\dots,4)$. In this case, the centroids $\mu_{jt}$ are far from the optimal cluster centroids $\mu_j^*$ and the contribution of each subproblem to the overall lower bound is small. The resulting bound is $\sum_{t=1}^T \text{MSSC}(S_t,K) = 102.86$, which leads to a large optimality gap of about $82\%$.

This simple example shows the need to carefully choose the anticluster partition by solving Problem \eqref{eq:cls}. However, Problem \eqref{eq:cls} can be solved exactly only for very small-sized instances; therefore, we define an \textit{ad hoc} heuristic for the problem. Note that we define Problem \eqref{eq:cls} using the restriction of the starting heuristic solution, without knowing whether it is optimal. Furthermore, even if the solution were optimal, Assumption \ref{ass0} may fail, so that its restriction may not be optimal. However, the lower bound in Eq. \eqref{eq:lower_weak} is valid for any partition $\mathcal{S}$. The partition obtained by solving Problems \eqref{eq:cls} seeks a partition with high intra-anticluster dissimilarity and low inter-anticluster dissimilarity, thereby improving the contribution of the corresponding MSSC subproblems to the overall bounds. As a result, the quality of the lower bound remains high even when Assumption \ref{ass0} does not hold and the input clustering is not optimal, as shown by the numerical experiments in Section \ref{sec:results}.

\begin{figure}[!ht]
    \centering
    \begin{subfigure}{0.4\textwidth}
        \caption{MSSC$(S_1, 4)= 25.715$}
        \includegraphics[scale=0.5]{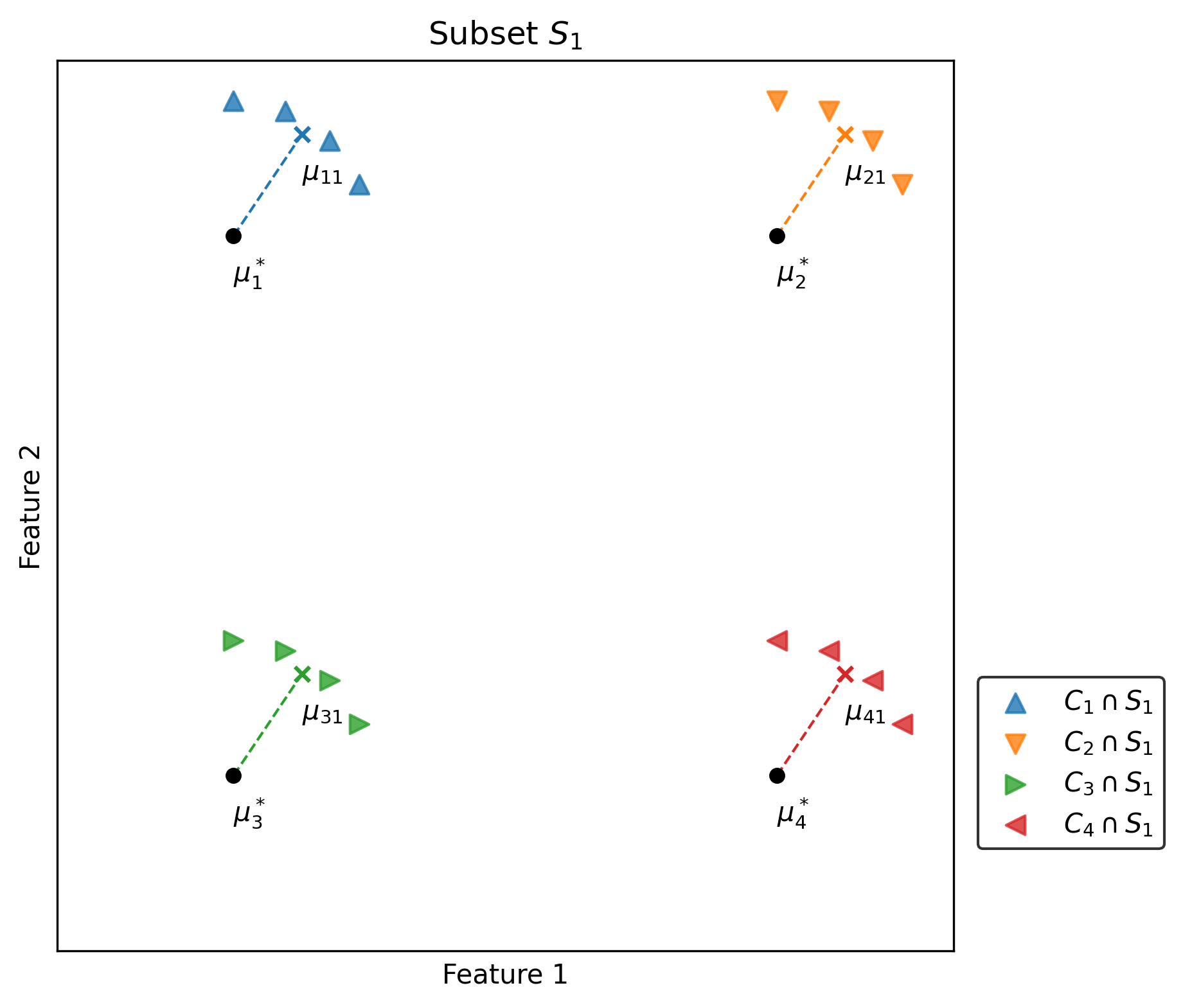}
        \label{fig:sol1w}
    \end{subfigure}
    \hfill
    \begin{subfigure}{0.4\textwidth}
        \caption{MSSC$(S_2, 4)= 25.715$}
        \includegraphics[scale=0.5]{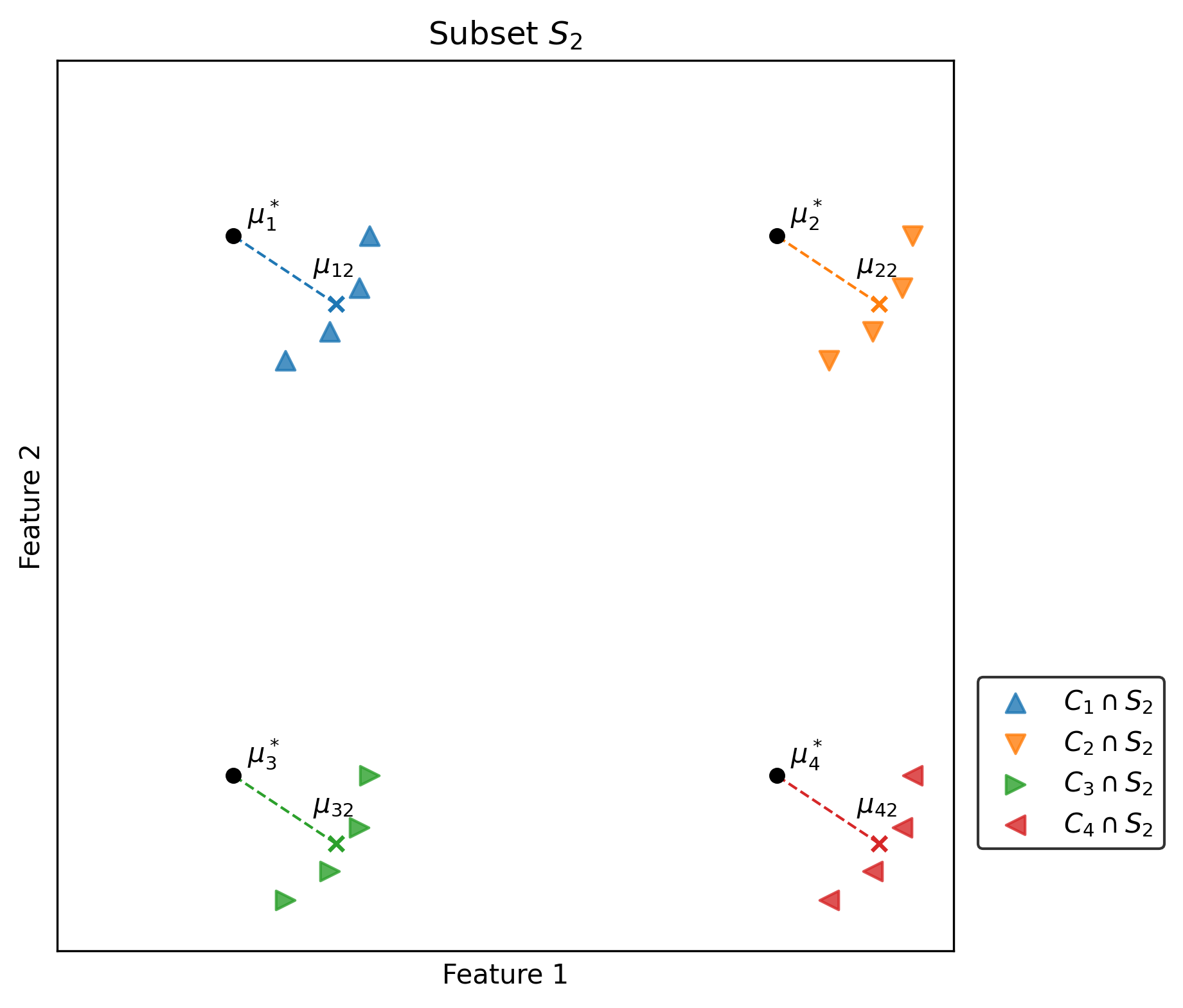}
        \label{fig:sol2w}
    \end{subfigure}\\
    \begin{subfigure}{0.4\textwidth}
        \centering
        \caption{MSSC$(S_3, 4)= 25.715$}
        \includegraphics[scale=0.5]{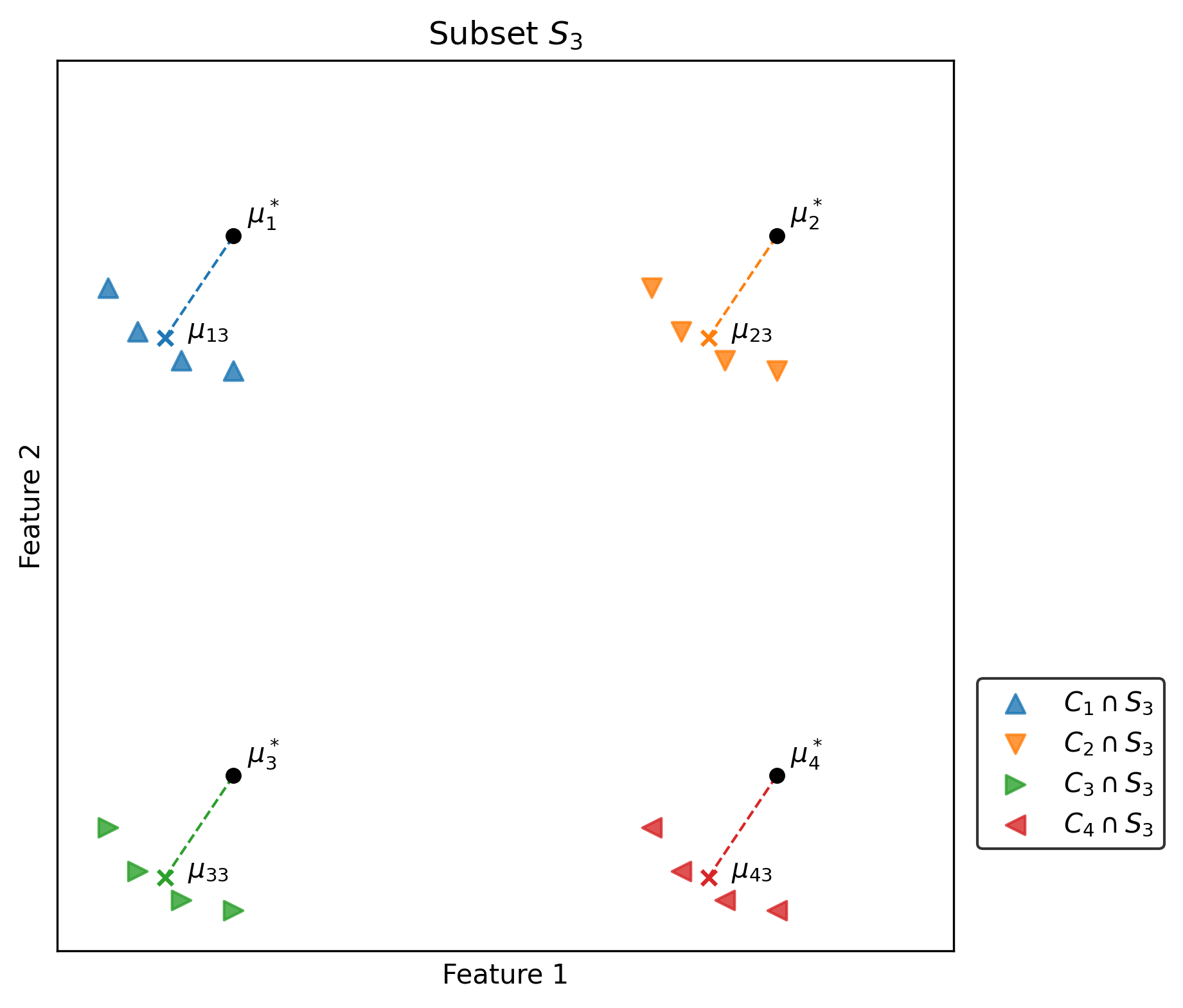}
        \label{fig:sol3w}
    \end{subfigure}
    \hfill
    \begin{subfigure}{0.4\textwidth}
        \centering
        \caption{MSSC$(S_4, 4)= 25.715$}
        \includegraphics[scale=0.5]{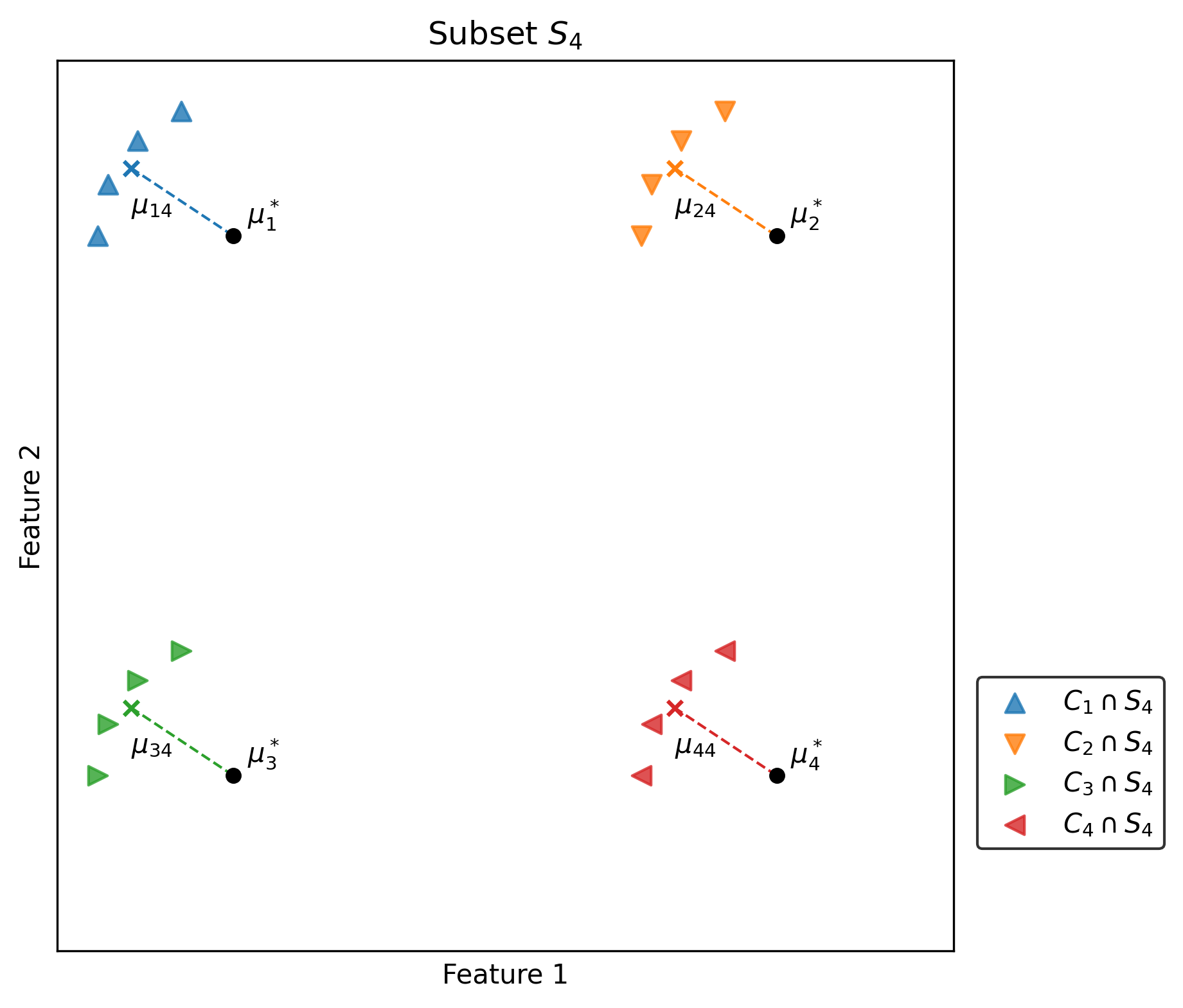}
        \label{fig:sol4w}
    \end{subfigure}
    \caption{A clustering partition for a dataset of $N = 64$ points, with $K = 4$ clusters and $T=4$ anticlusters. The optimal value satisfies $\textrm{MSSC}(O,K) =  575.674$, $\sum_{t=1}^T$MSSC$(S_t, K) = 102.86$ and the gap is $82\%$.}
    \label{fig:example}
\end{figure}

\section{\textsc{AVOC}: Anticlustering Validation Of Clustering}  \label{sec:avoc}
We propose a heuristic algorithm, named Anticlustering Validation Of Clustering (\textsc{AVOC}), that validates a given clustering partition $\mathcal{P} = \{C_1, \ldots, C_K\}$ of the clustering problem with $K$ clusters on a set of points $O$ with objective function value $UB:=\textrm{MSSC}_\textrm{UB}(O, K)$.
The purpose of \textsc{AVOC} is to assess the quality of feasible clustering solutions by computing a valid lower bound and the corresponding optimality gap, thereby providing an a posteriori validation even for large-scale instances where exact optimization is impractical.
The method operates through the iterative refinement of an initial anticlustering partition $\mathcal{S} = \{S_1, \dots, S_T\}$, using a swap procedure designed to maximize the objective of Problem~\eqref{eq:lui}.
Specifically, the algorithm includes the following key components: (i) an initialization procedure to produce a feasible anticlustering partition; (ii) an evaluation mechanism that produces an approximation $LB^+$ for the value $LB$; (iii) a heuristic mechanism that iteratively refines the current anticlustering partition to improve the value of $LB^+$; (iv) an algorithm that, given an anticlustering partition, produces a lower bound on the value $\textrm{MSSC}(O,K)$. A schematic representation of the \textsc{AVOC} algorithm is given in Figure \ref{fig:flow}. The pseudocode is outlined in Algorithm 1.

\begin{figure}[]
    \centering
    \scalebox{0.65}{
    \begin{tikzpicture}[scale=0.5]
      \node [io] (input) {{Input clustering $\mathcal{P}$ with value $UB$, centroids $\mu_j$ for $j \in [K]$, minimum gap $\epsilon_\gamma$,\\ number of starts $R$,\\
      number of anticlusters $T$}};
    
      \node [block, below of=input, node distance=2.5cm] (init) {Initialize candidate partitions\\
      $\mathcal{S}^{(r)}=\{S^{(r)}_1,\ldots,S^{(r)}_T\}$ for $r\in[R]$};
    
      \node [block, below of=init, node distance=2.5cm] (selectbest) {Evaluate each candidate $\mathcal{S}^{(r)}$ using\\
      $LB^+(\mathcal{S}^{(r)},K)=\sum_{t=1}^{T} LB^+(S^{(r)}_t,K)$\\
      Select best candidate $\mathcal{S}$ such that\\
      $r^\star=\arg\max_r LB^+(\mathcal{S}^{(r)},K)$};
    
      \node [block, below of=selectbest, node distance=2.8cm] (seed) {
      Set $LB^+=LB^+(\mathcal{S}^{(r^\star)},K)$ and $\,\gamma^+=\frac{UB-LB^+}{UB}$};
    
      \node [decision, below of=seed, node distance=3cm] (evaluate) {$\gamma^+ \leq \epsilon_{\gamma}$\, ?};
    
      \node [decision, below of=evaluate, node distance=3cm] (time) {Time limit?};
    
      \node [block, below of=time, node distance=3cm] (ubsol) {Create a new partition $\mathcal{S}'=\{S'_1,\ldots,S'_T\}$\\ via swap procedure};
    
      \node [block, below of=ubsol, node distance=2.5cm] (differs) {Evaluate $\mathcal{S}'$\\
      Compute $LB^+(S'_t,K)$ for each $t\in[T]$};
    
      \node [decision2, below of=differs, node distance=3cm] (update) {$\sum_{t=1}^{T} LB^+(S'_t,K) > LB^+$\, ?};
    
      \node [block, left of=update, node distance=9cm] (evaluate2) {
        Set $LB^+=\sum_{t=1}^{T} LB^+(S'_t,K)$\\
        Set $\mathcal{S}=\mathcal{S}'$ and update $\gamma^+$\\
      };
    
      \node [block, right of=update, node distance=9cm] (compute) {Compute for each $t\in[T]$\\
      MSSC$_\textrm{LB}(S_t,K)$};
    
      \node [start, below of=compute, node distance=3cm] (stop) {Return\\[1.5ex]
      $\gamma_\textrm{\tiny LB} = \frac{UB - \sum_{t=1}^T \textrm{MSSC}_\textrm{LB}(S_t,K)}{UB}$};
    
      \path [line] (input) -- (init);
      \path [line] (init) -- (selectbest);
      \path [line] (selectbest) -- (seed);
      \path [line] (seed) -- (evaluate);
      \path [line] (evaluate) -- node[label=\textsc{\scriptsize NO}, yshift=-0.7em, xshift=-1em] {} (time);
      \path [line] (time) -- node[label=\textsc{\scriptsize NO}, yshift=-0.7em, xshift=-1em] {} (ubsol);
      \path [line] (time) -| node[label=\textsc{\scriptsize YES}, yshift=-0.2em, xshift=-12em] {} (compute);
      \path [line] (ubsol) -- (differs);
      \path [line] (differs) -- (update);
      \path [line] (update) -- node[label=\textsc{\scriptsize YES}, yshift=-0.2em, xshift=0em] {} (evaluate2);
      \path [line] (update) -- node[label=\textsc{\scriptsize NO}, yshift=-0.2em, xshift=0em] {} (compute);
      \path [line] (evaluate2) |- (evaluate);
      \path [line] (evaluate) -| node[label=\textsc{\scriptsize YES}, yshift=-0.2em, xshift=-12em] {} (compute);
      \path [line] (compute) -- (stop);
    
      \node[right of=selectbest, node distance=4cm]{};
      \node[above of=init]{};
    
    \end{tikzpicture}
    }
    \caption{Flowchart of the \textsc{AVOC} algorithm.}
    \label{fig:flow}
\end{figure}

\subsection{Initialize the \textsc{AVOC} algorithm}

Given a clustering partition $\mathcal{P}$, we generate an initial anticlustering partition $\mathcal{S} = \{S_1, \dots, S_T\}$. The process begins by randomly distributing an equal number of points from each cluster $j \in [K]$, that is $|C_j|/T$ points, into subsets $S_{jt}$ for $t \in [T]$. These subsets are then combined to form an initial anticlustering partition $\mathcal{S}$. To ensure feasibility, each anticluster $S_t$ must include exactly one subset from each cluster.
If Assumption \ref{ass0} holds, all possible combinations of the subset will lead to the same lower bound value. However, if Assumption \ref{ass0} does not hold, different combinations of subsets can affect the quality of the resulting $LB^+$ and hence of $LB$. For this reason, a mixed-integer linear programming (MILP) model is used to determine the optimal combination. The model reads as follows:
\begin{subequations}
\label{eq:mount}
\begin{align}
\label{obj} \max & \sum_{t=1}^{T} \sum_{m=1}^{T} \sum_{m'=m}^{T} \sum_{j=1}^{K} \sum_{j'=j+1}^{K} d_{jj'mm'} \, y^{t}_{jj'mm'}  \\
\label{c:one} \text{s.t.}~ & \sum_{t=1}^T x_{jm}^{t} = 1,\quad \forall j \in [K], \, \forall m \in [T],\\
\label{c:two} & \sum_{m=1}^T x_{jm}^{t} = 1,\quad \forall j \in [K], \, \forall t \in [T],\\
\label{c:three} & y^{t}_{jj'mm'} \leq x_{jm}^{t}, \quad \forall j, j' \in [K], \, \, j<j', \ \forall m,m',t \in [T], \ m \leq m',\\
& y^{t}_{jj'mm'} \leq x_{j'm'}^{t}, \quad \forall j, j' \in [K], \ j<j', \ \forall m,m',t \in [T], \ m \leq m',\\
\label{c:four} & y^{t}_{jj'mm'} \geq x_{jm}^{t} + x_{j'm'}^{t} - 1, \quad \forall j, j' \in [K], \ j<j', \ \forall m,m',t \in [T], \ m \leq m',\\
& x_{jm}^{t} \in \{0,1\}, \quad \forall j \in [K], \ \forall  m,t\in [T], \\
& y^{t}_{jj'mm'} \in \{0,1\}, \quad \forall j, j' \in [K], \ j<j', \ \forall m,m',t \in [T], \ m \leq m'.
\end{align}
\end{subequations}
The binary variable $x_{jm}^{t}$ is equal to 1 if subset $S_{jm}$ is assigned to anticluster $t$, and 0 otherwise. Furthermore, the binary variable $y^{t}_{jj'mm'}$ is equal to 1 if the subsets $S_{jm}$ and $S_{j'm'}$, with $j < j'$, are assigned to anticluster $t$, and 0 otherwise. Since we want to form anticlusters with data points from different clusters as far away as possible, the objective function \eqref{obj} maximizes the between-cluster distance for the subsets that belong to the same anticluster.
The distance between two subsets $S_{jm}$ and $S_{j'm'}$, $\forall j, j' \in [K], \, j < j'$ and $\forall m, m' \in [T]$, is computed as:
\[
d_{jj'mm'} = \sum_{p_i \in S_{jm}} \sum_{p_{i'} \in S_{j'm'}} \|p_i - p_{i'}\|^2.
\]
Constraints \eqref{c:one} state that each subset $S_{jm}$ must be assigned to exactly one anticluster; constraints \eqref{c:two} ensure that for each anticluster exactly one subset of points from cluster $j$ is assigned.
Constraints \eqref{c:two}-\eqref{c:three} help track the set of points $S_{jm}$ and $S_{j'm'}$ that are in the same anticluster $m$. Each feasible solution $\bar{x}^t_{jm}$ for all $j\in[K], m,t\in [T]$ of Problem \eqref{eq:mount} represents an anticlustering partition that can be obtained by setting $S_{t} = \bigcup_{j \in [K], \ m \in [T] \, : \,  \bar{x}^t_{jm}=1} S_{jm} \,$ for every $t \in [T]$.

To further improve the initialization, we adopt a multi-start procedure: the initialization step is repeated $R$ times, each time producing a feasible anticlustering partition $\mathcal{S}^{(r)}=\{S^{(r)}_1,\ldots,S^{(r)}_T\}$, with $r\in[R]$, according to the rules described above. These candidate partitions are then given as input to the evaluation phase.

\subsection{Evaluate an Anticlustering Partition}\label{sec:evaluate}

To evaluate an anticlustering partition $\mathcal{S}$ using the lower bound MSSC$_\textrm{LB}(\mathcal{S},K)$ it is necessary to either solve to optimality or compute a valid lower bound for the optimal value of Problem \eqref{eq:MSSC_intdata} with $K$ clusters for each anticluster $S_t$. However, since this approach can be computationally expensive (potentially requiring the solution of SDPs), we propose an alternative method to approximate the lower bound efficiently.

Given an anticlustering partition $\mathcal{S}$ and a clustering partition $\mathcal{P}$ characterized by centroids $\mu_1, \dots, \mu_K$, the $k$-means algorithm is applied to each anticluster $S_t$ in $\mathcal{S}$, using $\mu_1, \dots, \mu_K$ as the initial centroids.
For each $S_t$, this process produces a clustering partition $\mathcal{P}_t = \{C_1^t, \dots, C_K^t\}$ with value $LB^+(S_t, K)$, where the set $C_j^t$ denotes the points of cluster $j$ in the anticluster set $S_t$. An upper approximation on the lower bound MSSC$_\textrm{LB}(\mathcal{S},K)$ is then given by
$$
LB^+(\mathcal{S}) = \sum_{t=1}^T LB^+(S_t, K) \, .
$$
Among the candidate partitions generated in the initialization phase, the one with the largest value of $LB^+(\mathcal{S}^{(r)})$ is retained as the initial solution. We set
$\mathcal{S} = \mathcal{S}^{(r^\ast)}$, where $r^\ast = \arg\max_{r\in[R]} LB^+(\mathcal{S}^{(r)})$, and define $LB^+ = LB^+(\mathcal{S}^{(r^\ast)})$.
The best candidate provides the starting point for the swap procedure (see Section~\ref{sec:swap}), {and yields the following estimate of the optimality gap:
\begin{equation}\label{eq:gamma+}
    \gamma^+ = \frac{UB - LB^+}{UB}\,,
\end{equation}
which is used to guide the \textsc{AVOC} procedure.
}

\subsection{Apply a Swap Procedure}\label{sec:swap}

Given an initial anticlustering partition $\mathcal{S} = \{S_1, \dots, S_T\}$ {with value $LB^+$} and clustering partitions $\mathcal{P}_t = \{C_1^t, \dots, C_K^t\}$ for every $t \in [T]$, we aim to improve $\mathcal{S}$ by swapping data points between different anticlusters.
A new anticlustering partition $\mathcal{S}'$ is then created by swapping two points $p_i \in C_j^t$ and $p_{i'} \in C_j^{t'}$, with $t \neq t'$, where both points belong to the same cluster $j \in [K]$ but are located in different anticlusters.

To guide this process, for each cluster $j \in [K]$, we rank the anticlusters based on their contribution to the current value $LB^+$. These contributions, denoted as $LB^+_{tj}$, are calculated using the intra-cluster distances of points within $C_j^t$ as follows:
$$
LB^+_{tj} = \sum_{i:p_i \in C^t_j} \sum_{i':p_{i'} \in C^t_j} \frac{\|p_i - p_{i'}\|^2}{2 \, \cdot |C_j^t|} .
$$
Furthermore, in each subset $C_j^t$, we order the points by their distance from the centroid $\mu_j$ of the cluster $j$ in the original clustering partition $\mathcal{P}$. This ordering allows us to prioritize swaps that are more likely to improve partition quality.

The algorithm proceeds by considering swaps between the points in the anticluster with the lowest $LB^+_{tj}$ which are closest to $\mu_j$, and the points in the anticluster with the highest $LB^+_{t'j}$ which are furthest from $\mu_j$.
This approach serves two purposes: balancing contributions across anticlusters and maximizing the potential impact of the swap. By swapping a distant point with a low-contributing point, the algorithm can effectively increase the objective function $LB^+$ while maintaining balanced contributions  $LB^+(S_t, K)$ across anticlusters.
Given a cluster $j \in [K]$ and a pair of points such that $p_i \in C^{t}_j$ and $p_{i'} \in C^{t'}_j$, a new candidate anticlustering partition $\mathcal{S}' = \{S'_1, \dots, S'_T\}$ is constructed as follows: 
\begin{itemize}
    \item $S'_m = S_m$, for all $m\in [T]\setminus \{t, t'\}$
    \item $S'_t = S_t \setminus \{p_{i}\} \cup \{p_{i'}\}$
    \item $S'_{t'} = S_{t'} \setminus \{p_{i'}\} \cup \{p_{i}\}$
\end{itemize}
The new partition $\mathcal{S}'$ is then evaluated (see Section \ref{sec:evaluate}).
If the swap results in a higher solution value, i.e., $\sum_{t=1}^T LB^+(S'_t, K) > LB^+$, the current partition and  lower bound estimate are updated:
$$\mathcal{S} = \mathcal{S}' \qquad \textrm{and} \qquad LB^+ = \sum_{t=1}^T LB^+(S_t, K)$$
and the gap $\gamma^+$ is recomputed.
The iterative procedure continues until one of the following conditions is met: the solution gap $\gamma^+$ reaches the desired threshold $\epsilon_{\gamma}$, a fixed time limit is reached ({\small T/O}), or no improvement is achieved after evaluating all possible swaps (\textsc{H}).

\begin{algorithm}
\small
\caption{Anticlustering Validation of Clustering (\textsc{AVOC})}
\KwIn{Clustering $\mathcal{P} = \{C_1, \dots, C_K\}$ with value $UB := \textrm{MSSC}_\textrm{UB}(O, K)$, centroids $\mu_1, \dots, \mu_K$, minimum gap $\epsilon_{\gamma}$, number of candidate initial partitions $R$, number of anticlusters $T$}
\KwOut{Solution gaps $\gamma_\textrm{\tiny LB}$}

\bigskip
{\textbf{Generate R candidate initial partitions} $\mathcal{S}^{(r)}$ with values $LB^+(\mathcal{S}^{(r)})$, $r\in[R]$}\;
{\textbf{Select the best candidate} $\mathcal{S} = \mathcal{S}^{(r^*)}$ where $r^* = \arg\max_r LB^+(\mathcal{S}^{(r)})$}\;
{\textbf{Set} $LB^+ = LB^+(\mathcal{S}^{(r^*)})$, $\gamma^+ = \frac{UB - LB^+}{UB}$ and set $\mathcal{P}_t = \{C_1^t, \ldots, C_K^t\}, \forall\, t \in [T]$}\;

\bigskip
\textbf{Apply a swap procedure}\;
\Repeat{$\gamma^+ \leq \epsilon_{\gamma}$, a time limit is reached, or no improvement is observed}{
\ForEach{cluster $j \in [K]$}{
    Sort anticlusters $t \in [T]$ by $LB^+_{jt}$ in a non-decreasing order, forming list $L_a$\;
    Sort anticlusters $t \in [T]$ by $LB^+_{jt}$ in a non-increasing order, forming list $L_d$\;

    \ForEach{anticluster $t$ in $L_a$}{
        Sort data points $p_i \in C^{t}_j$ in non-decreasing order of distance from $\mu_{j}$, forming list $D_a$\;
        \ForEach{data point $p_i$ in $D_a$}{
            \ForEach{anticluster $t'$ in $L_d \, | \, t \neq t'$}{
                Sort data points $p_{i'} \in C^{t'}_j$ in non-increasing order of distance from $\mu_{j}$, forming list $D_d$\;
                \ForEach{data point $p_{i'}$ in $D_d$}{
                    Construct a candidate partition $\mathcal{S}'$ by swapping $p_i$ and $p_{i'}$ as follows:
                    \begin{itemize}
                        \item $S'_m = S_m$ for $m \in [T]\setminus\{t,t'\}$
                        \item $S'_t = S_t \setminus \{p_{i}\} \cup \{p_{i'}\}$
                        \item $S'_{t'} = S_{t'} \setminus \{p_{i'}\} \cup \{p_{i}\}$
                    \end{itemize}
                    Evaluate $\mathcal{S'}$\;
                    Obtain for every $t \in [T]$ $\mathcal{P}'_t$ and value $LB^+(S'_t, K)$\;
                    \If{$\sum_{t=1}^T LB^+(S'_t, K)  \geq LB^+ $}{
                    Set $LB^+ = \sum_{t=1}^T LB^+(S'_t, K)$\;
                    Set $\mathcal{S} = \mathcal{S}'$ and update $\mathcal{P}_t$ for each $t \in [T]$ accordingly\;
                    Update gap $\gamma^+ = \frac{UB - LB^+}{UB}$\;
                    \textbf{break} to process the next cluster\;
                    }
                }
            }
        }
    }
}
}

\bigskip
\textbf{Produce a lower bound}\;
Compute  MSSC$_\textrm{LB}(S_t, K)$ for each $t \in [T]$\;
Set $\gamma_\textrm{\tiny LB} = \frac{UB - \sum_{t=1}^T \textrm{MSSC}_\textrm{LB}(S_t, K)}{UB}$\;
\bigskip

\textbf{Return} $\gamma_\textrm{\tiny LB}$\;
\medskip
\newpage
\end{algorithm}

\subsection{Produce a lower bound}
Given a clustering problem with a solution value $UB$, we aim to validate the $UB$ by leveraging the final anticlustering partition $\mathcal{S} = \{S_1, \ldots, S_T\}$ to compute a lower bound.
This is achieved by solving the clustering problem \eqref{eq:MSSC} for each anticluster set $S_t$, with $t \in [T]$, and obtaining either the optimal solution value $\textrm{MSSC}(S_t, K)$ or a lower bound on it, i.e., $\textrm{MSSC}_\textrm{LB}(S_t, K)$.
We define the gap as
{
\begin{equation}\label{eq:gammalb}
    \gamma_\textrm{\tiny LB} = \frac{UB - \sum_{t=1}^T \textrm{MSSC}_\textrm{LB}(S_t, K)}{UB} \, ,
\end{equation}
which provides a valid optimality certificate for the original problem.}

\section{Computational Results}  \label{sec:results}

This section provides the implementation details and presents the numerical results of the \textsc{AVOC} algorithm applied to synthetic and real-world datasets.

\subsection{Implementation Details}

The experiments are performed on a macOS system equipped with an Apple M4 Max chip (14-core) and 36 GB of RAM, running macOS version 15.5. The \textsc{AVOC} algorithm is implemented in C++.
To compute the initial clustering partition and perform the evaluation procedure of the \textsc{AVOC} algorithm (see Section \ref{sec:evaluate}), { we use the HG-means algorithm by \cite{gribel2019hg}, which is among the best-performing heuristics for MSSC.}
We use Gurobi v12.0 with all default settings to solve the MILP \eqref{eq:mount} {and a time limit of 60 seconds}.
In the initialization phase, we generate $R=15$ candidate partitions.
For the swap procedure, we set a minimum gap threshold $\epsilon_{\gamma}$ of 0.001\% and a maximum execution time set to $4 \cdot T$ minutes, where $T$ is the number of anticlusters chosen.

The lower bound computation leverages parallel processing in a multi-threaded environment using a pool of threads. Each clustering problem on an anticlustering set is processed in a separate thread. 
To compute the bound, we use \textsc{SOS}-\textsc{SDP}\footnote{\url{https://github.com/antoniosudoso/sos-sdp}}, the state-of-the-art exact solver for MSSC. 
For our experiments, the SDPs were solved only at the root node of the search tree, with a maximum of 80 cutting-plane iterations. An instance is solved successfully if the optimality gap is less than or equal to $10^{-4}$. This gap, defined as $(UB - LB) / UB$, measures the relative difference between the best upper bound (UB) and lower bound (LB). All other parameters are kept at their default values, as detailed in \cite{piccialli2022sos}.
The source code of \textsc{AVOC} {and the instances tested are} publicly available at \url{https://github.com/AnnaLivia/AVOC}.

\subsection{Datasets}

\paragraph{Synthetic instances}
To illustrate the effectiveness of our algorithm on synthetic instances, we generate large-scale Gaussian datasets comprising $N=10,000$ data points in a two-dimensional space ($D=2$). These datasets vary in the number of clusters ($K \in \{2, 3, 4\}$) and noise levels. Specifically, the data points are sampled from a mixture of $K$ Gaussian distributions $\mathcal{N}(\mu_j, \Sigma_j)$  for $j \in \{1, \dots, K\}$, with equal mixing proportions. Each distribution is characterized by a mean $\mu_j$ and a shared spherical covariance matrix $\Sigma_j = \sigma I$, where the standard deviation $\sigma$ takes values in $\{0.50, 0.75, 1.00\}$ to represent different noise levels. The cluster centers $\mu_j$ are drawn uniformly from the interval $[-1.00, 10]$. Instances are labelled using the format $N$-$K$-$\sigma$. Figures \ref{fig:art1}-\ref{fig:art3} show  the datasets generated with $\sigma \in\{0.50,0.75,1.00\}$ for $K=3$. We can see that the clusters are well separated for $\sigma=0.5$, and become more confused when $\sigma$ increases.
\begin{figure}[ht]
    \centering
    \begin{subfigure}{0.3\textwidth}
        \includegraphics[width=\textwidth]{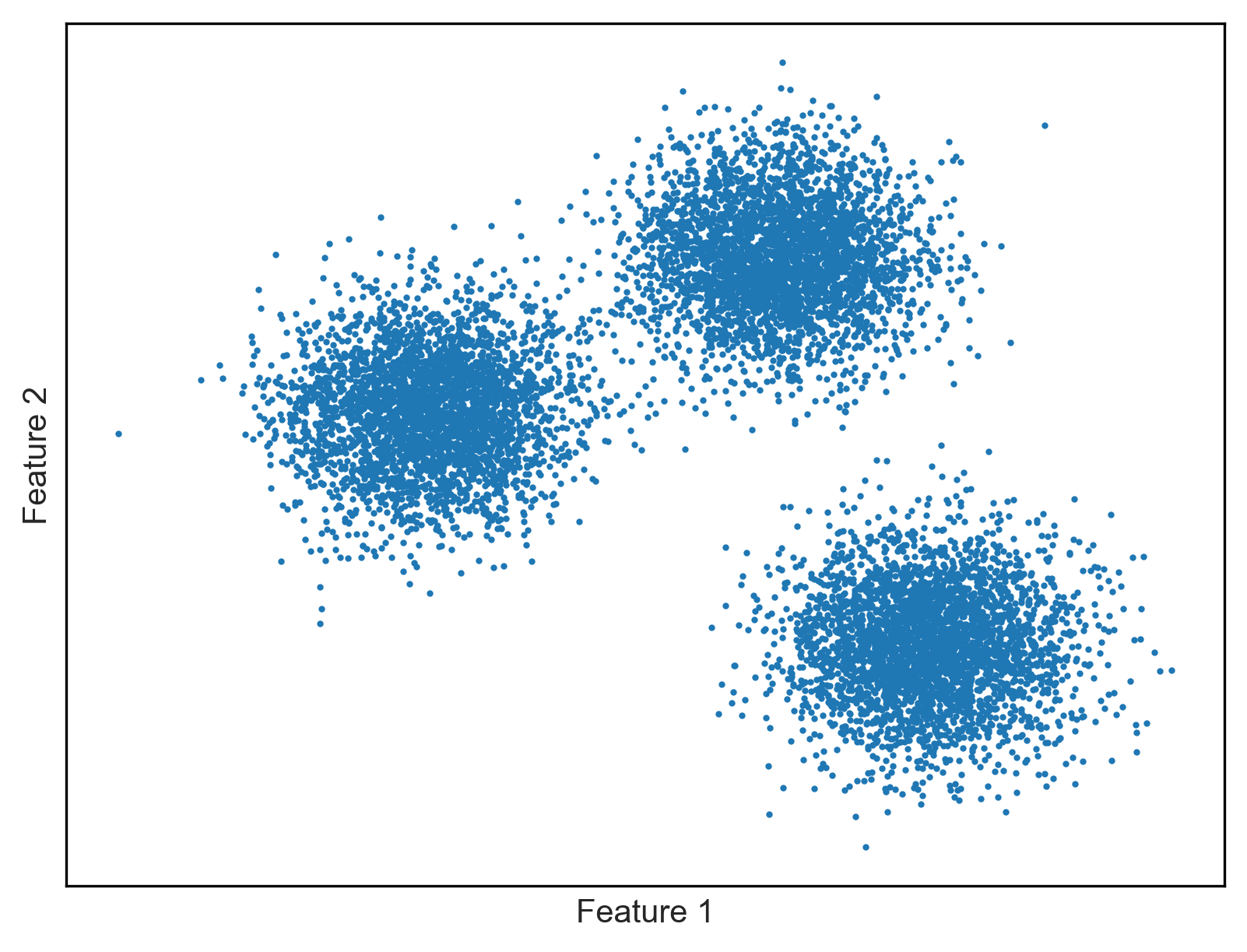}
        \caption{Dataset $10,000$-$3$-$0.50$}
        \label{fig:art1}
    \end{subfigure}
    \begin{subfigure}{0.3\textwidth}
        \includegraphics[width=\textwidth]{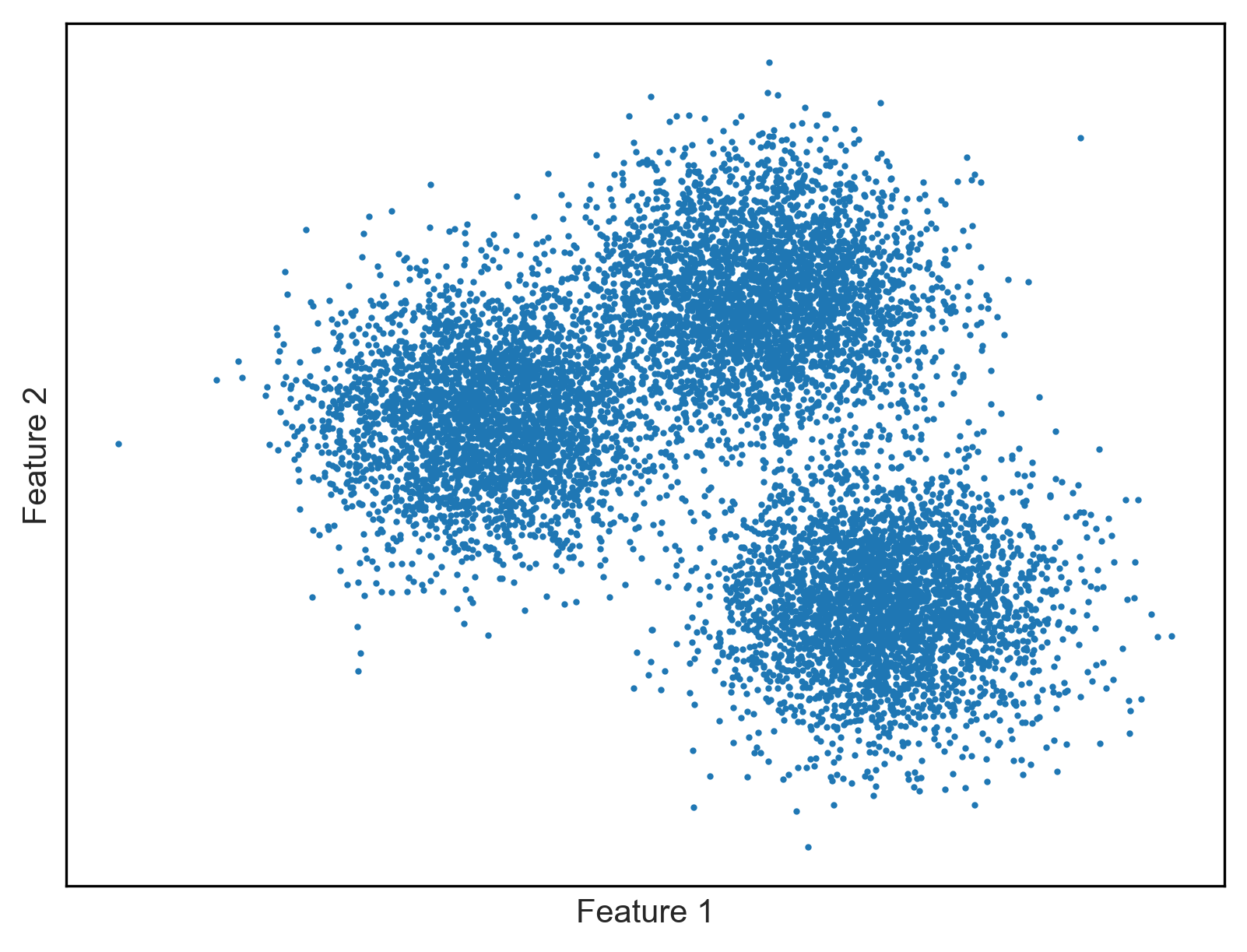}
        \caption{Dataset $10,000$-$3$-$0.75$}
        \label{fig:art2}
    \end{subfigure}
    \begin{subfigure}{0.3\textwidth}
        \centering
        \includegraphics[width=\textwidth]{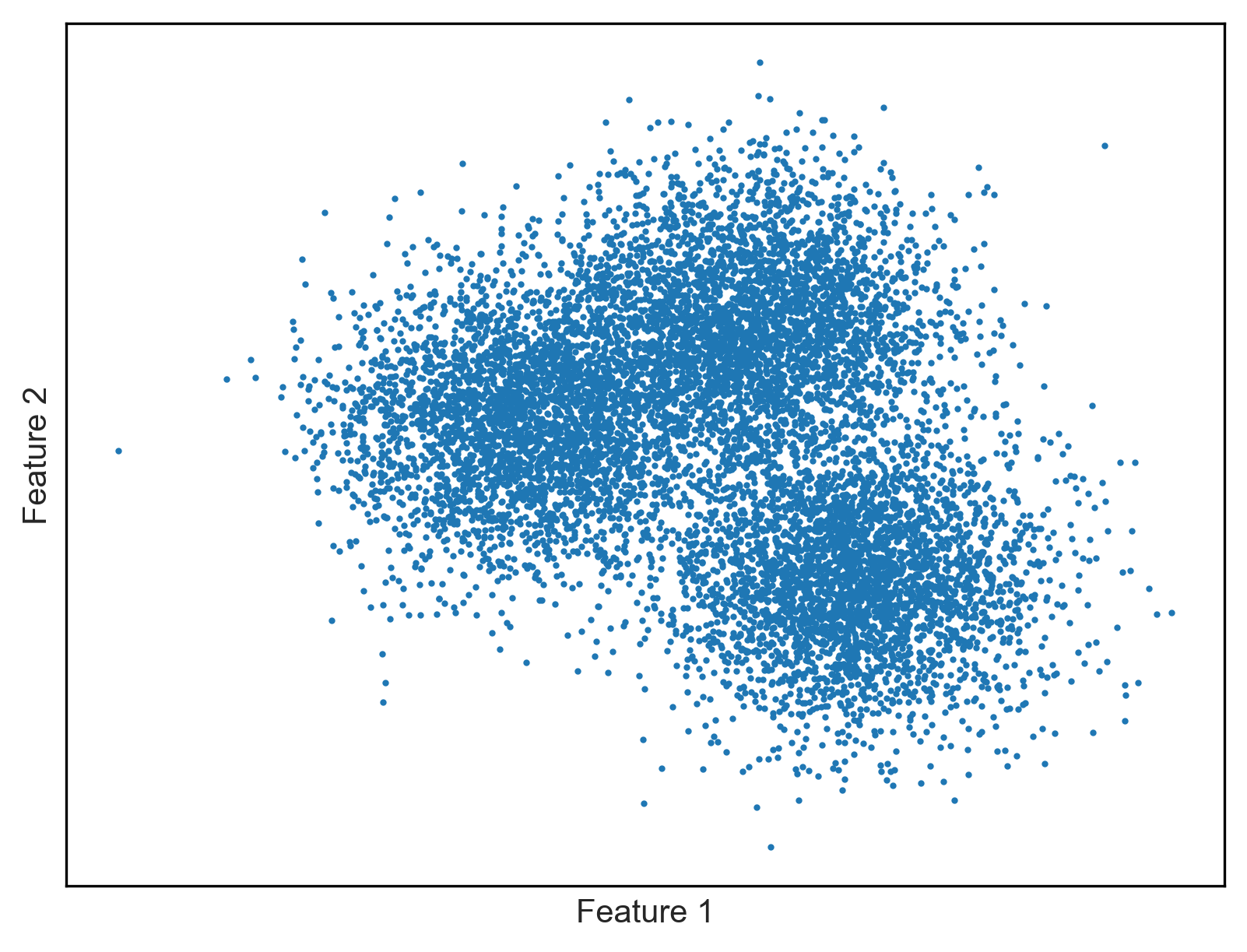}
        \caption{Dataset $10,000$-$3$-$1.00$}
        \label{fig:art3}
    \end{subfigure}
    \caption{Visualization of the synthetic datasets generated with number of data points $N=10,000$, number of clusters $K = 3$, and noise level $\sigma \in \{0.50, 0.75, 1.00\}$.}
    \label{fig:art}
\end{figure}

\paragraph{Real-world instances}
We consider { eight} real-world datasets { with varying sample size and dimensionality}. All of them can be downloaded from UCI\footnote{\url{https://archive.ics.uci.edu/}} and {UCR\footnote{\url{https://www.cs.ucr.edu/~eamonn/time\_series\_data\_2018/}} repositories}. The number of clusters is chosen with the elbow method. Table \ref{tab:1} reports { the value of the clustering partition $UB:=\textrm{MSSC}_\textrm{UB}(O,K)$ and} the datasets characteristics, i.e., the number of data points $N$, features $D$, clusters $K$ and the number of points in each cluster in the initial clustering partition $\mathcal{P}$. The smallest instance (\textit{Abalone}) is the largest instance solved to optimality by \textsc{SOS}-\textsc{SDP} in \cite{piccialli2022sos}.

\begin{table}[htbp]
\centering
\small
\begin{tabular}{l|rrrr|rrrr}
\textbf{{Dataset}} & {$UB$} & $N$ & $D$ & $K$ & \multicolumn{4}{c}{$|C_1| \, \dots \, |C_K|$}\\
\bottomrule
\textit{Abalone} & 1,005 & 4,177 & 10 & 3 & 1,308 & 1,341 & 1,528 & \\
\textit{Electric} & 101,769 & 10,000 & 12 & 3 & 2,886 & 3,537 & 3,577 & \\
\textit{Facebook} & 62,469 & 7,050 & 13 & 3 & 218 & 2,558 & 4,274 & \\
\textit{Frogs} & 84,437 & 7,195 & 22 & 4 & 605 & 670 & 2,367 & 3,553\\
\textit{Pulsar} & 92,209 & 17,898 & 8 & 2 & 2,057 & 15,841 & &\\
{\textit{StarLightCurves}} & { 3,476,012} &{ 9,236} & {1,024} & {3} & {2,552} & {2,576} & {4,108} & \\
{\textit{TwoPatterns}} & { 548,616} & {5,000} & {128} & {4} & {1,118} & {1,159} & {1,296} & {1,427}\\
{\textit{Wafer}} & { 439,777} & {7,164} & {152} & {2} & {2,401} & {4,763} & & \\
\bottomrule
\end{tabular}
\caption{Characteristics of real-world datasets.}
\label{tab:1}
\end{table}

\subsection{Results}

Each dataset has been tested with different values of the number of anticlusters $T$. For every dataset, we considered five values of $T$, selected according to the number of data points $N$ and the size of the clusters in the initial solution.
The choice of $T$ is guided by two main requirements: (i) each anticluster must remain tractable, i.e., contain fewer than 1,000 data points; and (ii) each cluster must be sufficiently represented within every anticluster, ensuring that no cluster is underrepresented.
For synthetic datasets, where clusters are balanced, we consistently used $T=\{10,12,15,17,20\}$.
For real-world datasets, the choice of $T$ was tailored to both the size of the instance and the distribution of cluster cardinalities. In the presence of unbalanced clusters, special care is required to ensure adequate representation. For example, in the \textit{Facebook} dataset, the smallest cluster contains 218 elements (see Table~\ref{tab:1}). Selecting $T > 16$ would result in fewer than 13 points from this cluster in each anticluster, potentially leading to poor representation.

For comparison, we consider two natural baselines for the anticlustering partition: \textit{(i)} a random equal-size partition of the dataset $O$, and \textit{(ii)} the best candidate solution obtained from the multi-start initialization phase, which serves as the starting point of \textsc{AVOC}. 
{
For each of these partitions, the corresponding optimality gap $\gamma_{\textrm{\tiny LB}}$ is computed according to \eqref{eq:gammalb}. These values are used as baselines to assess the improvement achieved by our algorithm, and are denoted by $\gamma_{\textrm{\tiny LB}}^{\textrm{rnd}}$ and $\gamma_{\textrm{\tiny LB}}^{\textrm{start}}$.
}
It is worth noting that the MILP model \eqref{eq:mount} is not used in constructing baseline \textit{(i)}, as it is an integral component of the \textsc{AVOC} procedure itself.

The results of the experiments on synthetic and real-world instances are presented in Tables \ref{tab:sint} and \ref{tab:real}, respectively.
For each experiment, the tables report: the number of partitions ($T$); 
the solutions gaps of the benchmark anticlustering partitions, namely the random baseline ($\gamma_\textrm{\tiny LB}^{\textrm{rnd}}$) and the initial solution ($\gamma_\textrm{\tiny LB}^{\textrm{start}}$); and
the lower solution gaps {in percentage} ($\gamma_\textrm{\tiny LB}${\tiny(\%)}). Computational times are also included, specifically the time in seconds required to initialize the algorithm ($Init$ {\tiny(s)}); the time spent in the swap procedure in seconds ($Heur$ {\tiny(s)}); the time spent by \textsc{SOS-SDP} to solve the root node of the clustering problems in seconds ($\textsc{SOS}$ {\tiny (s)}); and the total time for the \textsc{AVOC} algorithm in minutes ($Time${\tiny(min)}).
Furthermore, for real-world instances we report also the quality gap ($\gamma^+${\tiny(\%)}) in percentage, { computed according to \eqref{eq:gamma+}}, and the stopping criterion ($Stop$), which could be (a) achieving the minimum gap ($\epsilon_{\gamma}$), (b) reaching the time limit ({\scriptsize T/O}), or (c) observing no further improvement (\textsc{h}).
For the synthetic instances, we omit these columns as the swap procedure always terminates due to reaching the minimum gap criterion (\(\gamma^+ = \epsilon_{\gamma}\)).

To fully appreciate the results, it is important to note that solving an instance of around 1,000 data points to global optimality requires several hours of computational time. According to the computational results in \cite{piccialli2022sos}, solving a synthetic instance of size 3,000 with noise comparable to $\sigma=1$ takes around 10 hours, while solving a real dataset of approximately 4,000 datapoints requires more than 24 hours.  
Although the machine used in \cite{piccialli2022sos} and the one used for the experiments in this paper differ, the CPU times reported in \cite{piccialli2022sos} provide a useful reference for understanding the computational challenge of the instances considered here.

\paragraph{Synthetic instances}
\begin{table}[!ht]
    \centering
    \scriptsize
    \begin{tabular}{l|r|rrr|rrrr}
        \toprule
        Instance & $T$ & $\gamma_\textrm{\tiny LB}^{\textrm{rnd}}$ \tiny{(\%)}  & $\gamma_\textrm{\tiny LB}^{\textrm{start}}$ \tiny{(\%)} & $\gamma_\textrm{\tiny LB}${\tiny(\%)} & Init \tiny{(s)} & Heur \tiny{(s)} & \textsc{SOS} \tiny{(s)} & Time \tiny{(min)}\\
        \midrule
        \multirow{5}[2]{*}{10000\_2\_05} & 10    & 0.170& 0.131& 0.004& 2     & 17    & 67    & 1 \\
              & 12    & 0.245& 0.103& 0.002& 5     & 11    & 62    & 1 \\
              & 15    & 0.297& 0.182& 0.002& 16    & 14    & 56    & 1 \\
              & 17    & 0.270& 0.220& 0.002& 34    & 30    & 60    & 2 \\
              & 20    & 0.343& 0.302& \textbf{0.001} & 64    & 37    & 59    & 3 \\
        \midrule
        \multirow{5}[2]{*}{10000\_2\_075} & 10    & 0.244& 0.153& 0.026& 2     & 18    & 448   & 8 \\
              & 12    & 0.282& 0.170& 0.028& 5     & 14    & 323   & 6 \\
              & 15    & 0.329& 0.258& 0.023& 16    & 27    & 250   & 5 \\
              & 17    & 0.419& 0.251& 0.027& 31    & 18    & 217   & 4 \\
              & 20    & 0.402& 0.312& \textbf{0.022} & 62    & 25    & 199   & 5 \\
        \midrule
        \multirow{5}[2]{*}{10000\_2\_10} & 10    & 0.417& 0.301& 0.223& 2     & 17    & 535   & 9 \\
              & 12    & 0.459& 0.399& 0.172& 5     & 35    & 578   & 10 \\
              & 15    & 0.471& 0.403& 0.149& 16    & 29    & 360   & 7 \\
              & 17    & 0.501& 0.369& 0.101& 32    & 27    & 310   & 6 \\
              & 20    & 0.514& 0.345& \textbf{0.084} & 62    & 32    & 286   & 6 \\
        \midrule
        \multirow{5}[2]{*}{10000\_3\_05} & 10    & 0.278& 0.196& 0.005& 9     & 23    & 246   & 5 \\
              & 12    & 0.333& 0.242& 0.008& 60    & 22    & 251   & 6 \\
              & 15    & 0.393& 0.298& \textbf{0.003} & 159   & 64    & 133   & 6 \\
              & 17    & 0.446& 0.421& 0.010& 345   & 42    & 128   & 9 \\
              & 20    & 0.521& 0.478& 0.008& 559   & 60    & 133   & 13 \\
        \midrule
        \multirow{5}[2]{*}{10000\_3\_075} & 10    & 0.342& 0.258& 0.097& 18    & 29    & 559   & 10 \\
              & 12    & 0.451& 0.330& 0.053& 61    & 33    & 365   & 8 \\
              & 15    & 0.483& 0.413& 0.058& 159   & 45    & 298   & 8 \\
              & 17    & 0.579& 0.437& 0.046& 354   & 67    & 309   & 12 \\
              & 20    & 0.726& 0.504& \textbf{0.035} & 557   & 57    & 246   & 14 \\
        \midrule
        \multirow{5}[2]{*}{10000\_3\_10} & 10    & 1.427& 1.241& 1.026& 18    & 55    & 719   & 13 \\
              & 12    & 1.383& 1.163& 0.840& 58    & 47    & 922   & 17 \\
              & 15    & 1.238& 1.093& 0.812& 168   & 61    & 505   & 12 \\
              & 17    & 1.232& 1.078& 0.706& 361   & 59    & 461   & 15 \\
              & 20    & 1.259& 1.004& \textbf{0.490} & 559   & 68    & 448   & 18 \\
        \midrule
        \multirow{5}[2]{*}{10000\_4\_05} & 10    & 0.350& 0.342& \textbf{0.002} & 103   & 37    & 276   & 7 \\
              & 12    & 0.480& 0.343& 0.005& 281   & 42    & 182   & 8 \\
              & 15    & 0.499& 0.475& 0.004& 601   & 72    & 154   & 14 \\
              & 17    & 0.612& 0.560& 0.002& 601   & 69    & 131   & 13 \\
              & 20    & 0.703& 0.677& 0.005& 601   & 83    & 148   & 14 \\
        \midrule
        \multirow{5}[2]{*}{10000\_4\_075} & 10    & 0.499& 0.314& 0.048& 104   & 61    & 500   & 11 \\
              & 12    & 0.541& 0.481& 0.053& 285   & 58    & 373   & 12 \\
              & 15    & 0.596& 0.510& 0.044& 601   & 79    & 267   & 16 \\
              & 17    & 0.692& 0.623& \textbf{0.043} & 601   & 85    & 263   & 16 \\
              & 20    & 0.975& 0.774& 0.044& 601   & 94    & 247   & 16 \\
        \midrule
        \multirow{5}[2]{*}{10000\_4\_10} & 10    & 1.230& 1.162& 0.861& 103   & 54    & 770   & 15 \\
              & 12    & 1.219& 1.041& 0.605& 285   & 57    & 568   & 15 \\
              & 15    & 1.033& 1.032& 0.530& 601   & 92    & 471   & 19 \\
              & 17    & 1.069& 0.897& 0.416& 601   & 83    & 411   & 18 \\
              & 20    & 1.180& 0.929& \textbf{0.301} & 601   & 86    & 365   & 18 \\
        \bottomrule
    \end{tabular}
    \caption{Summary of experimental results for the synthetic instances, including the number of partitions ($T$), {solutions gaps of the random baseline ($\gamma_\textrm{\tiny LB}^\text{rnd}$) and the initial solution ($\gamma_\textrm{\tiny LB}^\text{start}$)}, lower solution gap ($\gamma_\textrm{\tiny LB}$), computational times ({initialization} time, heuristic time, and \textsc{SOS}-\textsc{SDP} time), and total algorithm runtime in minutes (\textit{Time}).}
    \label{tab:sint}
\end{table}
Table \ref{tab:sint} shows that the \textsc{AVOC} algorithm provides strong and valid lower bounds on synthetic instances.
On average, it achieves an average gap of \(0.18\%\) within 10 minutes, with the swap procedure requiring only 10–90 seconds to reach the target accuracy $\gamma^+ = \epsilon_\gamma = 0.001 \%$. 
The random partition baseline (column $\gamma_\textrm{\tiny LB}^\text{rnd}$) already delivers reasonably good bounds with an average gap of \(0.62\%\), while the best multi-start initialization (column $\gamma_\textrm{\tiny LB}^\text{best}$) further improves them to \(0.51\%\).
In all cases, \textsc{AVOC} consistently improves both baselines, yielding to a high-quality approximation \(LB^+\).

Datasets with well-separated clusters (e.g., $\sigma=0.5$) exhibit negligible deviations between anticluster centroids and original centroids.
Therefore, the centroids remain nearly identical across various anticlusters. Under these conditions, the resulting lower bound is tight, yielding optimality gaps {below $0.003\%$ in all cases.}
On the other hand, as $\sigma$ increases, the separation between clusters decreases (see Figure \ref{fig:art}).
{
Consequently, for instances with $\sigma > 0.5$, Assumption~1 may no longer hold, making the gap behaviour less predictable as $T$ varies. This effect is reflected in the results: as the noise level increases, the optimality gaps increase as well.
For larger noise levels, the gaps remain small but become more pronounced, reaching at most $1.026\%$ in the most challenging instance.
However, each value of $T$ yields a valid lower bound and, therefore, a valid optimality certificate. Considering, for each instance, the best gap obtained among the tested values of $T$, the optimality gap remains below $0.5\%$ for all synthetic datasets.
}
{
These empirical observations are consistent with the theoretical analysis in Section~\ref{sec:theoretical}.
When the noise level is low ($\sigma=0.5$), clusters are well separated, the restrictions are more likely to preserve the original clustering structure, and the induced centroids remain close to the original ones. Consequently, the resulting optimality gaps are extremely small. As the noise level increases, the cluster structure in each anticluster may deviate from the clustering solution one, increasing the distance between the anticluster centroids and the original centroids. This leads to larger gaps, as stated in Proposition~\ref{prop:tight}.
}
Finally, higher noise levels make the clustering subproblems more challenging, leading to a moderate increase in the computational time required by \textsc{SOS}-\textsc{SDP}.

\paragraph{Real-world instances}
\begin{table}[!ht]
    \centering
    \scriptsize
    \begin{tabular}{l|r|rrrr|rrrrr}
        \toprule
        Instance & $T$ & $\gamma_\textrm{\tiny LB}^\text{rnd}${\tiny(\%)} & $\gamma_\textrm{\tiny LB}^\text{start}${\tiny(\%)} & $\gamma^+$\tiny{(\%)} & $\gamma_\textrm{\tiny LB}$\tiny{(\%)} & Init \tiny{(s)} & Heur \tiny{(s)} & Stop & \textsc{SOS} \tiny{(s)} & Time \tiny{(min)}\\
        \midrule
        \multirow{5}[2]{*}{Abalone} & 4     & 0.359& 0.085& 0.001& \textbf{0.003} & 0     & 64    & $\epsilon_\gamma$   & 270   & 6 \\
              & 5     & 0.261& 0.136& 0.001& 0.004& 0     & 92    & $\epsilon_\gamma$   & 169   & 4 \\
              & 6     & 0.349& 0.244& 0.001& 0.004& 1     & 84    & $\epsilon_\gamma$   & 75    & 3 \\
              & 8     & 0.565& 0.238& 0.001& 0.005& 4     & 409   & \textsc{h} & 64    & 8 \\
              & 10    & 0.600& 0.419& 0.002& 0.007& 13    & 819   & \textsc{h} & 152   & 16 \\
        \midrule
        \multirow{5}[2]{*}{Electric} & 10    & 3.171& 3.056& 0.001& 3.018& 15    & 662   & $\epsilon_\gamma$   & 903   & 26 \\
              & 15    & 2.972& 2.935& 0.001& \textbf{2.091} & 202   & 1,122  & $\epsilon_\gamma$   & 1,267  & 43 \\
              & 20    & 2.832& 2.830& 0.001& 2.183& 688   & 3,983  & \textsc{h} & 794   & 91 \\
              & 25    & 3.235& 3.152& 0.002& 2.485& 887   & 4,312  & \textsc{h} & 587   & 96 \\
              & 30    & 3.538& 3.335& 0.003& 2.737& 903   & 3,800  & \textsc{h} & 511   & 87 \\
        \midrule
        \multirow{5}[2]{*}{Facebook} & 7     & 4.272& 2.803& 0.012& \textbf{2.354} & 3     & 1,681  & {\tiny T/O} & 1,075  & 46 \\
        & 8     & 5.751& 3.479& 0.012& 2.731& 5     & 1,952  & {\tiny T/O} & 993   & 49 \\
        & 10    & 7.107& 4.947& 0.028& 3.711& 13    & 2,437  & {\tiny T/O} & 952   & 57 \\
        & 13    & 8.046& 5.747& 0.098& 5.314& 148   & 3,156  & {\tiny T/O} & 807   & 69 \\
        & 16    & 10.040& 9.146& 0.203& 6.175& 890   & 3,841  & {\tiny T/O} & 706   & 91 \\
        \midrule
        \multirow{5}[2]{*}{Frogs} & 8     & 3.443& 2.800& 0.004& 2.615& 14    & 1,931  & {\tiny T/O} & 1,746  & 62 \\
          & 10    & 3.285& 3.069& 0.007& 2.577& 61    & 2,493  & {\tiny T/O} & 1,290  & 64 \\
          & 13    & 3.245& 2.609& 0.013& 2.069& 558   & 3,201  & {\tiny T/O} & 985   & 79 \\
          & 15    & 3.206& 2.827& 0.016& 1.994& 894   & 3,642  & {\tiny T/O} & 857   & 90 \\
          & 16    & 2.957& 2.747& 0.023& \textbf{1.722} & 901   & 3,929  & {\tiny T/O} & 790   & 94 \\
        \midrule
        \multirow{5}[2]{*}{Pulsar} & 18    & 3.260& 3.118& 0.002& 2.675& 56    & 4,393  & {\tiny T/O} & 4,432  & 148 \\
              & 20    & 3.118& 3.012& 0.001& 2.690& 149   & 4,849  & {\tiny T/O} & 3,850  & 147 \\
              & 25    & 3.312& 2.899& 0.001& 2.521& 553   & 6,067  & {\tiny T/O} & 2,491  & 152 \\
              & 30    & 2.933& 2.593& 0.002& 2.385& 830   & 7,350  & {\tiny T/O} & 1,960  & 169 \\
              & 35    & 2.926& 2.464& 0.003& \textbf{2.125} & 902   & 8,479  & {\tiny T/O} & 1,813  & 187 \\
        \midrule
    \multirow{5}[1]{*}{{StarLightCurves}} & {10} & {1.174} & {0.912} & {0.037} & {0.860} & {49} & {2,492} & {{\tiny T/O}} & {9,063} & {193} \\
          & {12} & {1.182} & {0.803} & {0.126} & {0.873} & {292} & {2,993} & {{\tiny T/O}} & {8,629} & {199} \\
          & {15} & {1.108} & {0.812} & {0.083} & {0.677} & {852} & {3,782} & {{\tiny T/O}} & {5,650} & {171} \\
          & {18} & {0.985} & {0.748} & {0.096} & {0.579} & {876} & {4,790} & {{\tiny T/O}} & {4,227} & {165} \\
          & {20} & {1.066} & {0.690} & {0.092} & {\textbf{0.445}} & {825} & {4,998} & {{\tiny T/O}} & {3,185} & {150} \\
        \midrule
        \multirow{5}[1]{*}{{TwoPatterns}} & {5} & {3.167} & {3.068} & {0.090} & {3.138} & {4} & {1,344} & {{\tiny T/O}} & {6,423} & {130} \\
              & {8} & {3.175} & {3.011} & {0.170} & {3.022} & {31} & {2,002} & {{\tiny T/O}} & {3,787} & {97} \\
              & {10} & {3.367} & {3.115} & {0.248} & {\textbf{2.863}} & {115} & {2,487} & {{\tiny T/O}} & {3,149} & {96} \\
              & {12} & {3.552} & {3.152} & {0.271} & {2.940} & {801} & {2,894} & {{\tiny T/O}} & {2,149} & {97} \\
              & {15} & {3.892} & {3.426} & {0.230} & {3.234} & {905} & {4,070} & {{\tiny T/O}} & {2,704} & {128} \\
        \midrule
        \multirow{5}[1]{*}{{Wafer}} & {8} & {0.171} & {0.159} & {0.041} & {0.046} & {1} & {1,940} & {{\tiny T/O}} & {1,486} & {57} \\
      & {10} & {0.266} & {0.194} & {0.048} & {0.050} & {3} & {2,488} & {{\tiny T/O}} & {967} & {58} \\
      & {12} & {0.311} & {0.271} & {0.055} & {0.058} & {6} & {2,898} & {{\tiny T/O}} & {1,623} & {75} \\
      & {15} & {0.374} & {0.282} & {0.037} & {\textbf{0.040}} & {23} & {3,665} & {{\tiny T/O}} & {1,464} & {86} \\
      & {20} & {0.593} & {0.458} & {0.054} & {0.056} & {431} & {4,921} & {{\tiny T/O}} & {1,442} & {113} \\
        \bottomrule
    \end{tabular}
    
    \caption{Summary of experimental results for the real-world instances, including the number of partitions ($T$), {solutions gaps of the random baseline ($\gamma_\textrm{\tiny LB}^\text{rnd}$) and the initial solution ($\gamma_\textrm{\tiny LB}^\text{start}$)}, lower solution gap ($\gamma_\textrm{\tiny LB}$, quality gap ($\gamma^+$), computational times ({initialization} time, heuristic time, and \textsc{SOS}-\textsc{SDP} time), and total algorithm runtime in minutes (\textit{Time}).}
    \label{tab:real}
\end{table}
Table \ref{tab:real} highlights distinct behaviors across real-world datasets.
In all cases, the algorithm produces a better gap than both the random baseline $\gamma_\textrm{\tiny LB}^\text{rnd}$ and the multi-start initialization gap $\gamma_\textrm{\tiny LB}^\text{start}$. A consistent hierarchy emerges: random partitions yield the weakest bounds, initialization provides a noticeable improvement, and the refinement stage further reduces the gap. The magnitude of this reduction, however, depends on the dataset.
We note that, for each dataset, all experiments are performed starting from the same initial clustering partition. Therefore, the smallest gap obtained across the tested values of $T$ can be regarded as the strongest optimality certificate for that instance.

For smaller datasets such as \textit{Abalone}, the method achieves very tight bounds ($\gamma_{\textrm{\tiny LB}} \approx \gamma^+$), with $\gamma^+$ always converging to the target threshold $\epsilon_\gamma = 0.001\%$. These instances are solved within minutes (at most 16), providing results comparable to exact methods at a fraction of the computational cost. For comparison, \cite{piccialli2022sos} reports an exact solution in 2.6 hours, whereas here the same dataset is validated within about 10 minutes, with gaps ranging between $0.003\%$ and $0.007\%$ depending on the number of anticlusters.

On larger datasets such as \textit{Electric} and \textit{Pulsar}, \textsc{AVOC} still provides strong certificates of quality: the final lower bound gaps $\gamma_\textrm{\tiny LB}$ remain around $2\%$, while $\gamma^+$ is close to the stopping threshold. These results are obtained within 1--3 hours. For the remaining datasets, the algorithm reaches the time limit before convergence, but still provides meaningful validation, with gaps below $3\%$.
Note that, the method appears robust with respect to the dimensionality of the data: high-dimensional datasets such as \textit{StarLightCurves}, \textit{TwoPatterns}, and \textit{Wafer}, whose number of features is two orders of magnitude larger than that of the other instances, exhibit comparable performance in terms of both gap quality and computational time.

From a computational perspective, increasing $T$ reduces the size of the subproblems, making the \textsc{SOS}-\textsc{SDP} phase faster, while increasing the cost of the swap-based heuristic, which is typically the dominant component. As a result, the total runtime does not exhibit a monotone behavior with respect to $T$, reflecting the trade-off between these two phases. Convergence to a local maximum of the anticlustering problem is observed only for \textit{Abalone} and \textit{Electric} at larger values of $T$.

\begin{figure}[!ht]
    \centering
    \begin{subfigure}{0.4\textwidth}
        \includegraphics[width=0.97\linewidth]{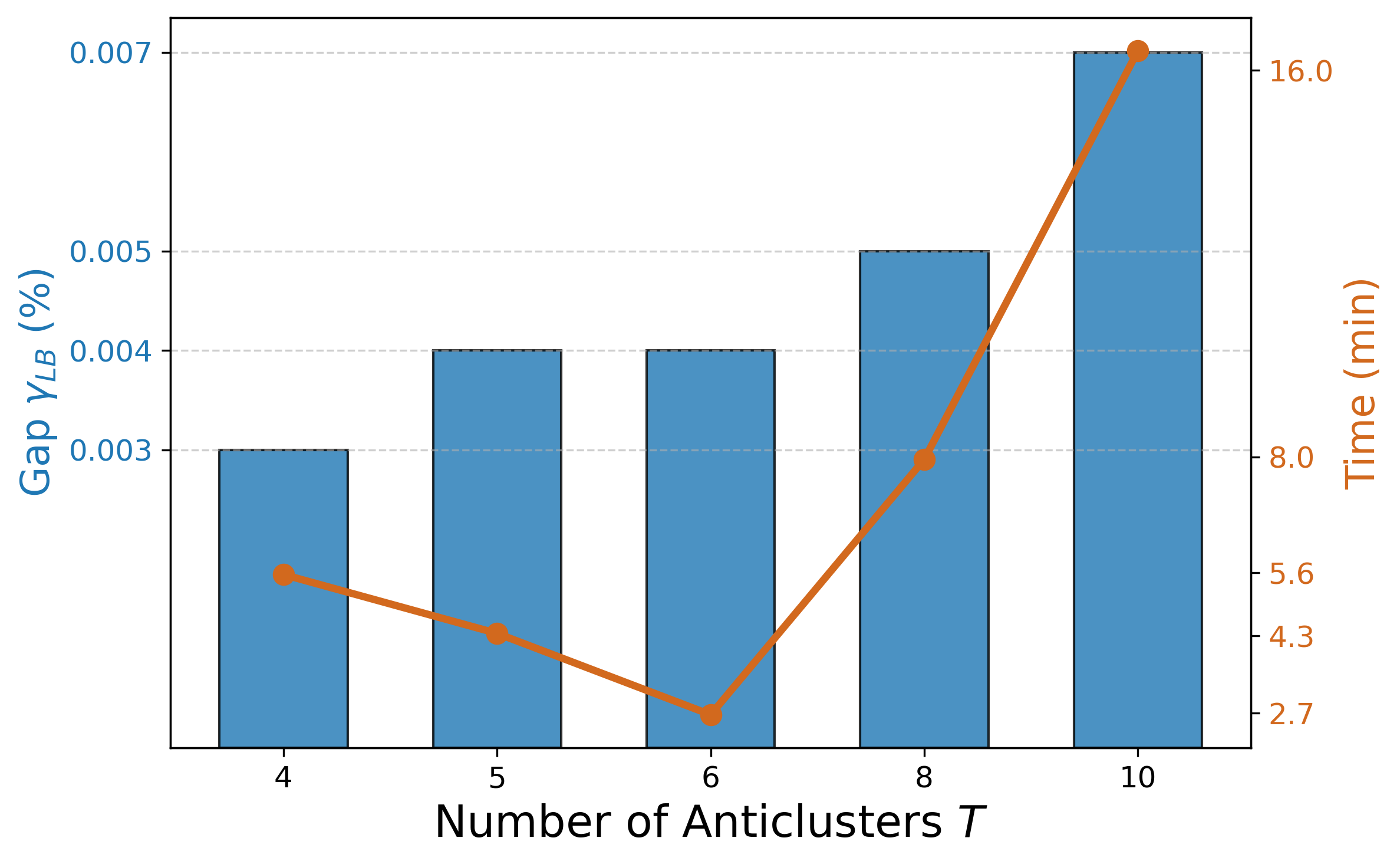}
        \caption{Abalone}
        \label{fig:abalone}
    \end{subfigure}
    \hspace{0.4cm}
    \begin{subfigure}{0.4\textwidth}
        \includegraphics[width=0.97\linewidth]{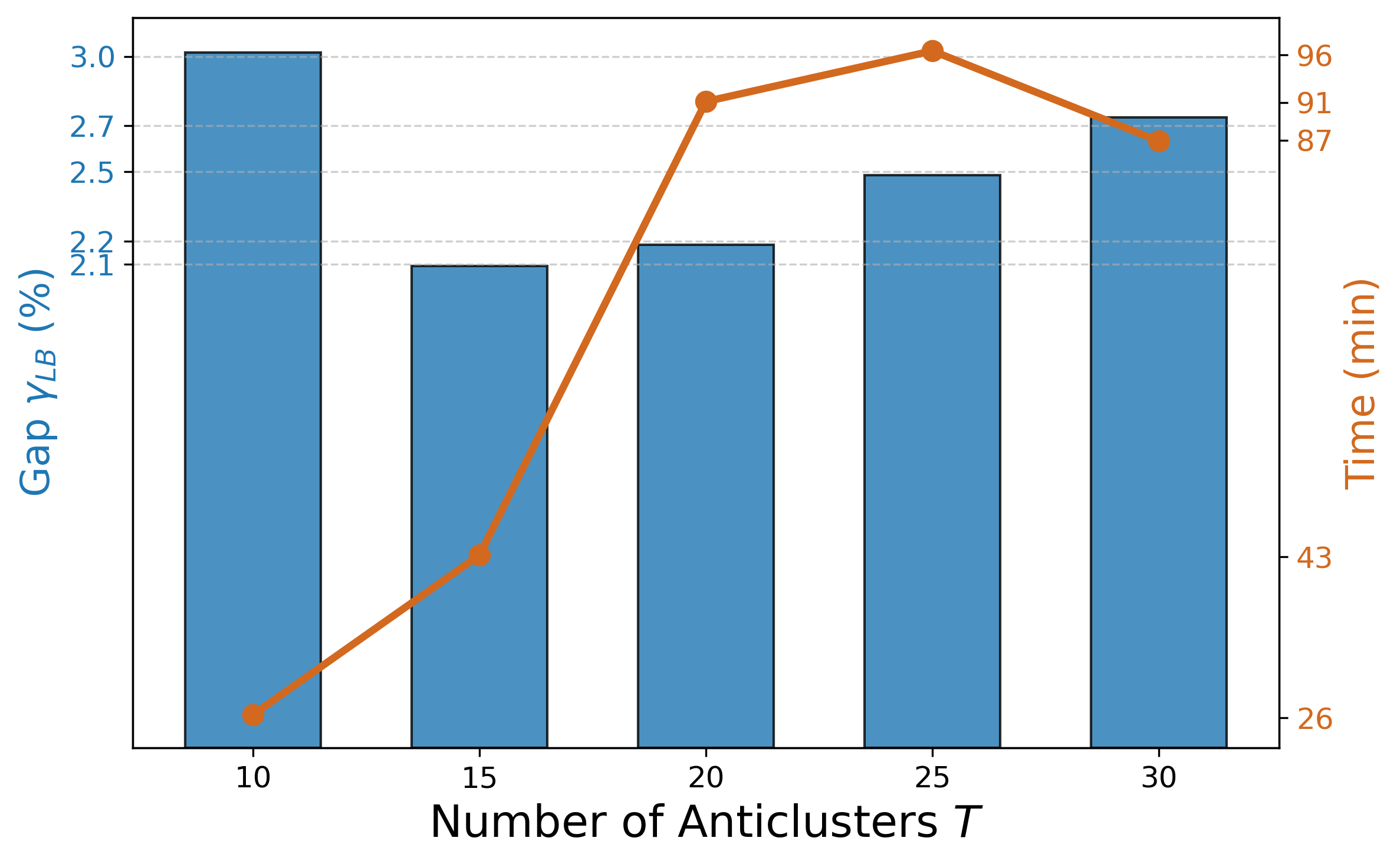}
        \caption{Electric}
        \label{fig:electric}
    \end{subfigure}\\[0.47cm]
    \begin{subfigure}{0.4\textwidth}
        \includegraphics[width=\linewidth]{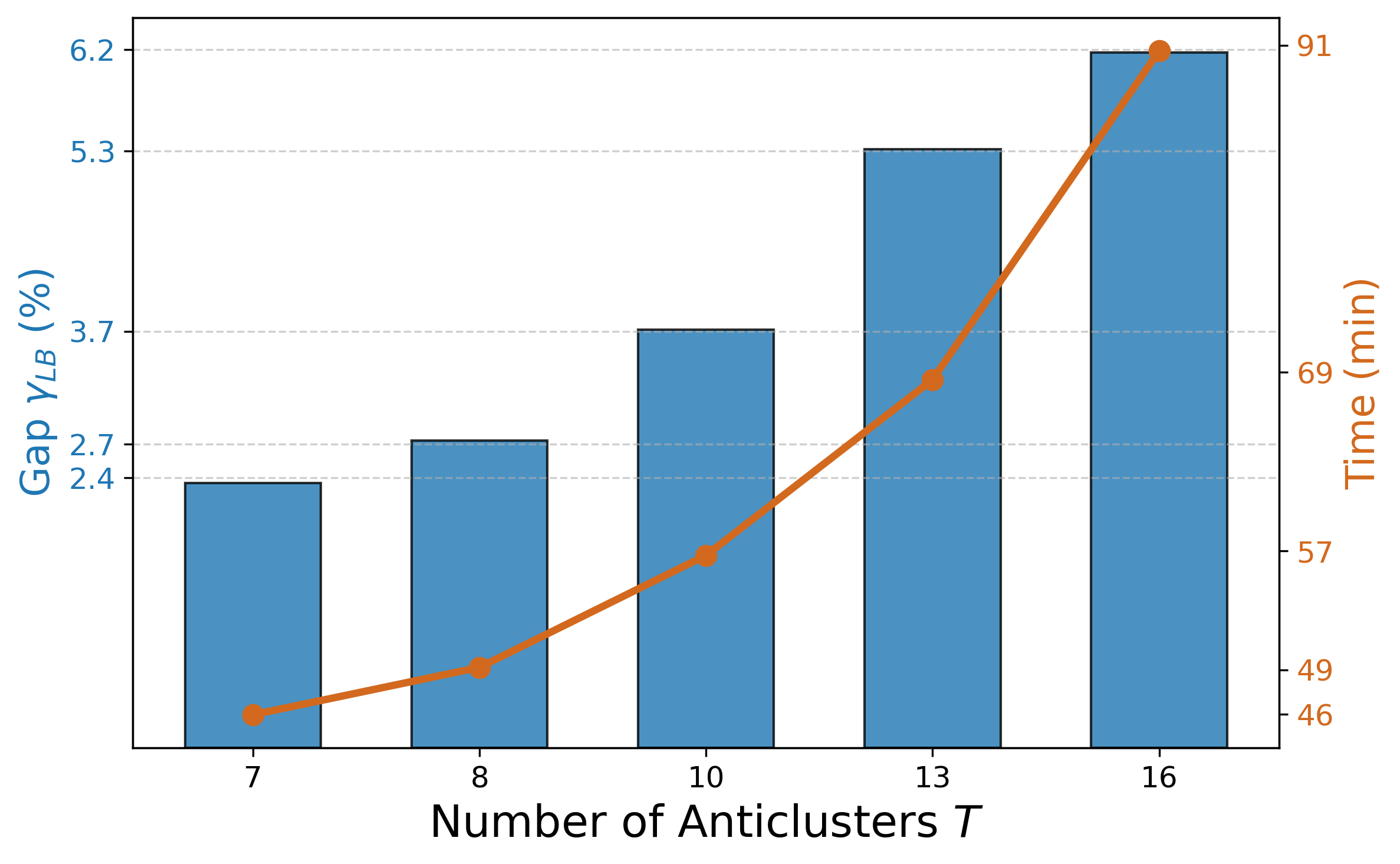}
        \caption{Facebook}
        \label{fig:facebook}
    \end{subfigure}
    \hspace{0.4cm}
    \begin{subfigure}{0.4\textwidth}
        \includegraphics[width=0.97\linewidth]{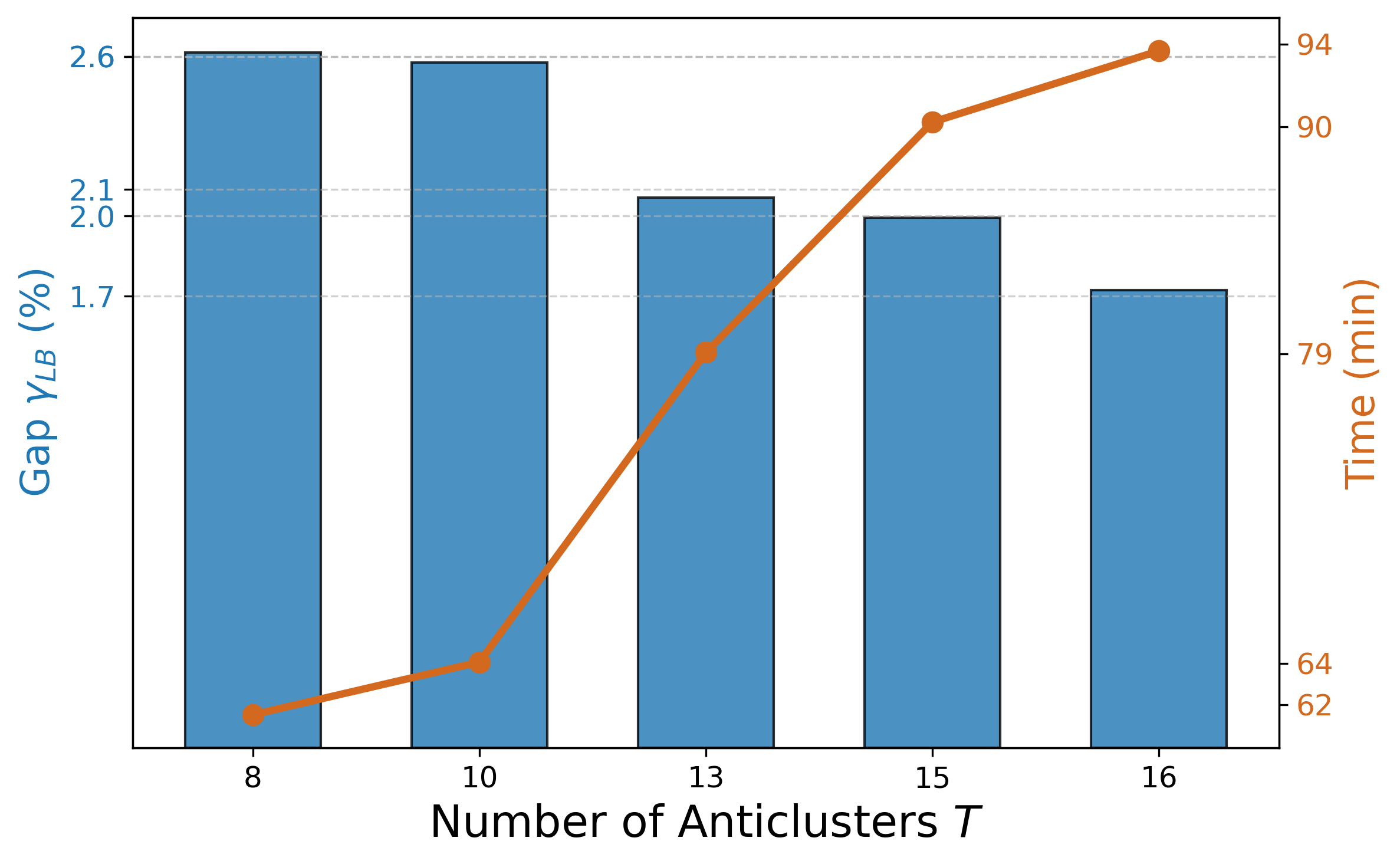}
        \caption{Frogs}
        \label{fig:frogs}
    \end{subfigure}\\[0.47cm]
    \begin{subfigure}{0.4\textwidth}
        \includegraphics[width=0.97\linewidth]{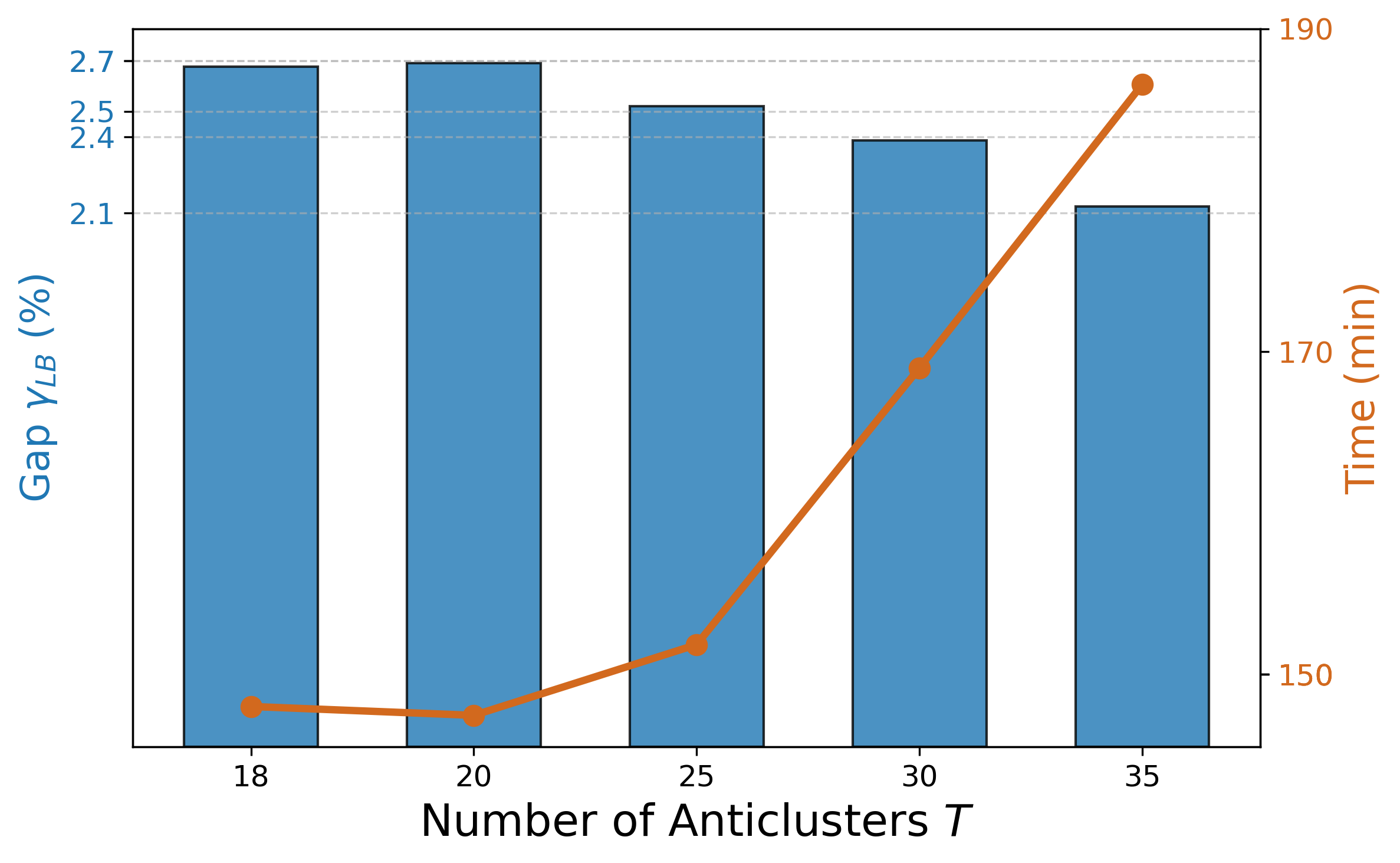}
        \caption{Pulsar}
        \label{fig:pulsar}
    \end{subfigure}
    \hspace{0.4cm}
    \begin{subfigure}{0.4\textwidth}
        \includegraphics[width=0.97\linewidth]{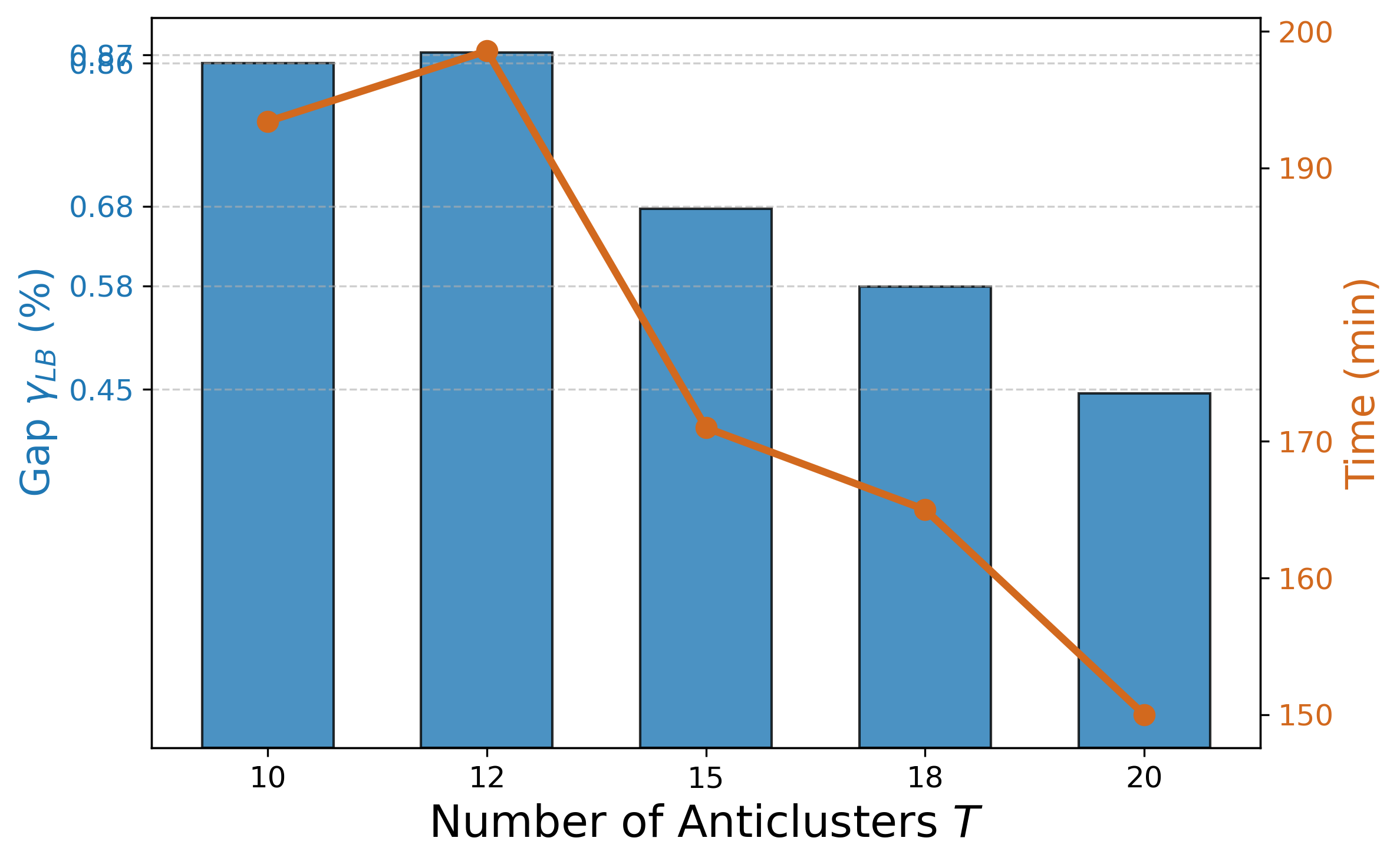}
        \caption{{StarLightCurves}}
        \label{fig:star}
    \end{subfigure}\\[0.47cm]
    \begin{subfigure}{0.4\textwidth}
        \includegraphics[width=\linewidth]{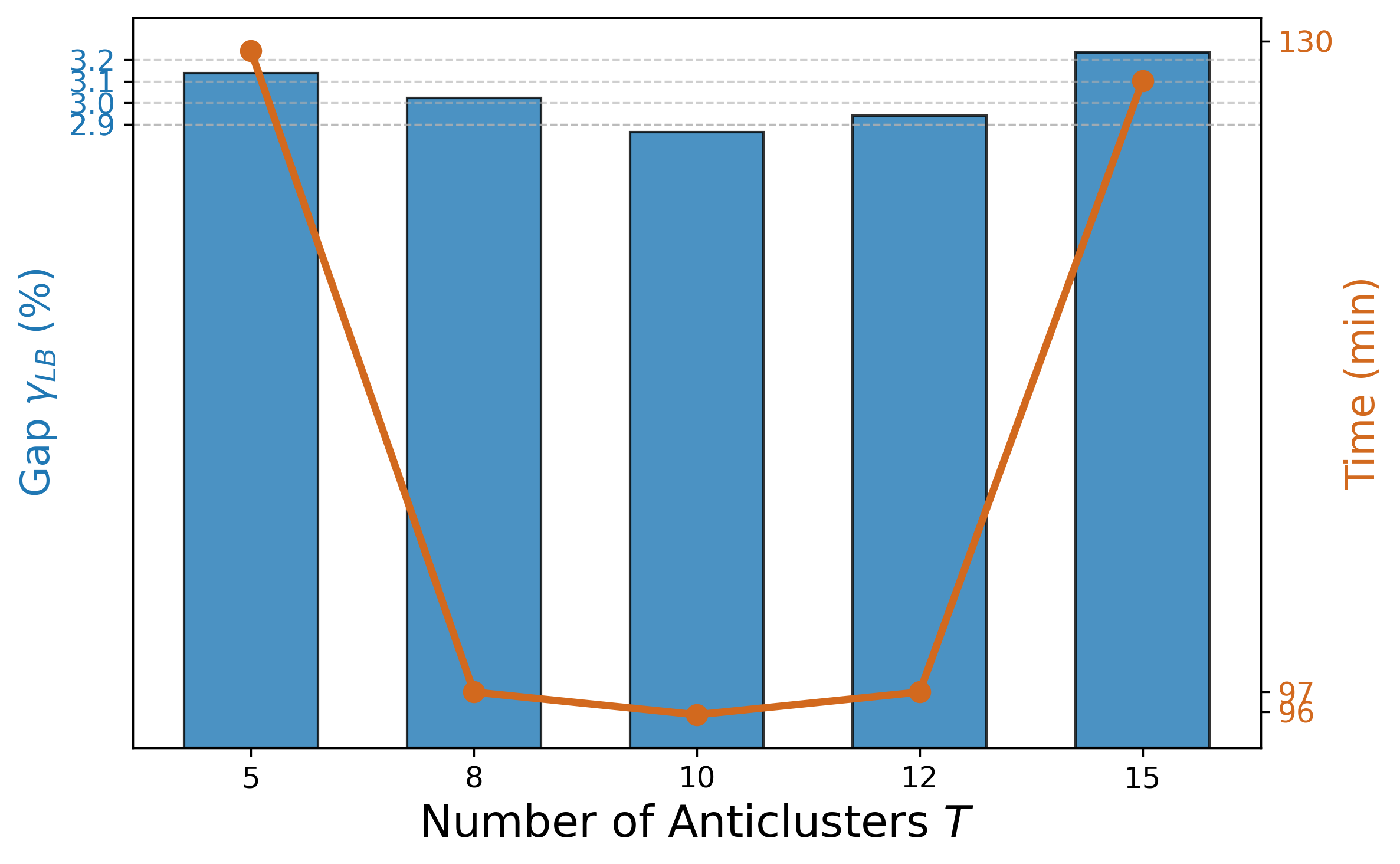}
        \caption{{TwoPatterns}}
        \label{fig:2patterns}
    \end{subfigure}
    \hspace{0.4cm}
    \begin{subfigure}{0.4\textwidth}
        \includegraphics[width=\linewidth]{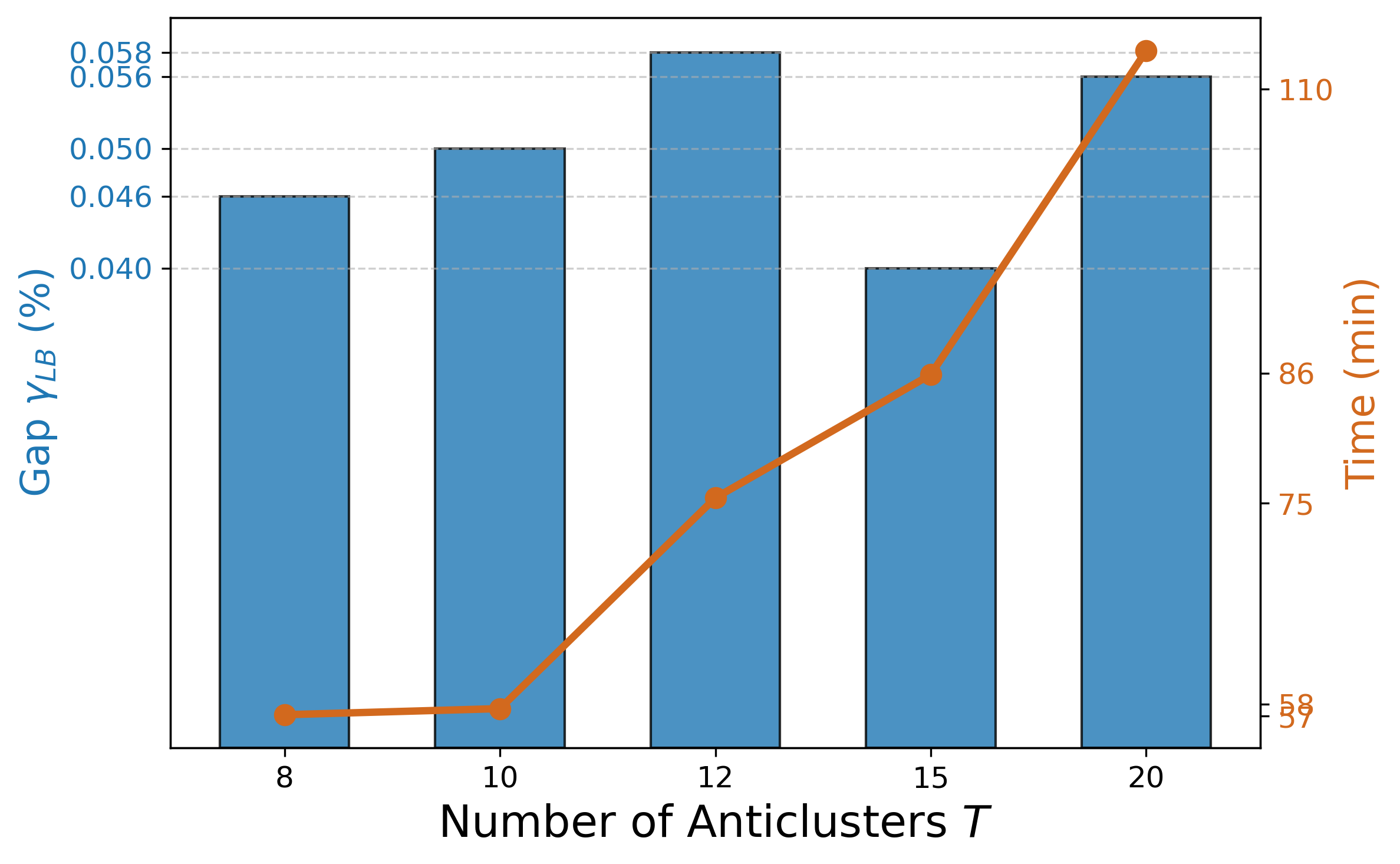}
        \caption{{Wafer}}
        \label{fig:wafer}
    \end{subfigure}
    \caption{{Performance comparison for different numbers of anticlusters. The bar chart represents the lower bound gap ($\gamma_{\textrm{\tiny LB}}$), while the orange line with markers indicates the total computation time in minutes.}}\label{fig:real_gaps}
\end{figure}

Figure~\ref{fig:real_gaps} further illustrates the relationship between $T$, the optimality gap $\gamma_{\textrm{\tiny LB}}$, and the total running time.
Across most datasets, the variation of $\gamma_{\textrm{\tiny LB}}$ with respect to $T$ is limited, indicating that the method is generally robust to the choice of the number of anticlusters.
{
In contrast, the \textit{Facebook} dataset shows a higher sensitivity to $T$ due to the poor representation of its smallest cluster within the anticlusters.
As $T$ increases, only a small number of the points in the smallest cluster is assigned to each anticluster, reducing the ability of the procedure to preserve the original clustering structure. Consequently, while the lower bound remains valid, its quality deteriorates.
Although other datasets such as \textit{Pulsar} also exhibit substantial cluster-size imbalance, their smallest clusters remain sufficiently represented in every anticluster for
the values of $T$ considered.
}

{
In summary, the choice of $T$ should balance tractability and representativeness: each anticluster should remain small enough to allow the efficient computation of the corresponding MSSC lower bound, while containing a sufficient number of points from every cluster so as to preserve the structure of the original clustering solution.
}

\section{Conclusions}  \label{sec:conclusions}
{In this paper, we present a novel methodology for validating feasible MSSC solutions, specifically designed to address a gap in the literature: the absence of certified optimality guarantees for large-scale instances. While the current state of the art relies heavily on heuristic algorithms to generate feasible solutions, these methods inherently only provide upper bounds on the optimal objective value. Our research introduces a framework that computes valid lower bounds, which are typically unavailable for large-scale instances.
We leverage the conceptual inverse of traditional clustering, i.e., the anticlustering problem, which involves partitioning a set of objects into groups characterized by high intra-group dissimilarity and low inter-group dissimilarity. This auxiliary problem serves as a guide for a divide-and-conquer strategy, allowing us to decompose the original, computationally intensive problem into smaller, more tractable subproblems. These subproblems are generated through a specialized heuristic method named AVOC.
The integration of any heuristic solution with our computed lower bound immediately produces a certified optimality gap.}

We conducted experiments on synthetic instances with 10,000 points and real-world datasets ranging from 4,000 to 18,000 points. The computational results demonstrate that the proposed method effectively provides an optimality measure for very large-scale instances. The optimality gaps, defined as the percentage difference between the MSSC of the clustering solution and the lower bound on the optimal value, remain below 3\% in all cases. However, many instances exhibit significantly smaller gaps, while computational times remain manageable, averaging around two hours. To the best of our knowledge, no other method exists that can produce (small) optimality gaps on large-scale instances.

While our approach can in principle be applied to datasets of any size, the current implementation has been tested on instances with up to about 18,000 points. Scaling to datasets with hundreds of thousands or millions of observations remains challenging, as the tractability of the subproblems becomes the limiting factor. In such settings, a large number of anticlusters $T$ is required to keep each subproblem manageable, which may in turn deteriorate the quality of the lower bound, since each subproblem contributes less effectively to the global bound. Future improvements in this direction may come from more powerful exact solvers for medium-sized MSSC instances and from enhanced parallelization strategies.
{
A further limitation of the proposed methodology lies in the dependence of the bound quality on how well the clustering structure is preserved across the anticlusters. As highlighted by the theoretical analysis in Section~3, the gap is driven by the discrepancy between the global centroids and those induced within the anticlusters. When clusters are well separated, this discrepancy is small, and the resulting bounds are tight. However, in more challenging settings, such as datasets with overlapping clusters or high noise levels, this condition may not hold, leading to weaker bounds. This effect is particularly critical in the presence of different cluster sizes, as smaller clusters may be poorly represented within the anticlusters. Although the validity of the lower bound is not affected, its quality may deteriorate in such cases.

These observations naturally suggest some directions for future research. A primary objective is to further enhance the AVOC procedure used to construct the anticlusters.}
In addition, it is important to better understand the trade-off between the quality of the lower bound and the computational effort required to obtain it. Finally, extending the approach to constrained clustering settings represents a promising direction for broadening its applicability.

\section*{Acknowledgments}
The work of Anna Livia Croella and Veronica Piccialli has been supported by PNRR MUR project PE0000013-FAIR.
This manuscript reflects only the authors’ views and opinions, neither the European Union nor the European Commission can be considered responsible for them.

\section*{Declaration of interest}
The authors declare that they have no known competing financial interests or personal relationships that could have appeared to influence the work reported in this paper.

\small
\bibliography{biblio}

\end{document}